\documentclass[11pt]{article}

\usepackage[utf8]{inputenc}
\usepackage{amsmath, amssymb, amsthm, mathtools}
\usepackage{graphicx}
\usepackage{hyperref}
\usepackage{geometry}
\usepackage{cite}
\usepackage{tikz}
\usepackage{tikz-cd}
\usepackage{subcaption}
\usepackage{float}

\newtheorem{theorem}{Theorem}[section]

\newtheorem{lemma}[theorem]{Lemma}
\newtheorem{proposition}[theorem]{Proposition}
\newtheorem{corollary}[theorem]{Corollary}
\newtheorem{remark}[theorem]{Remark}

\newtheorem{problem}[theorem]{Problem}
\newtheorem{example}[theorem]{Example}

\newcommand{\R}{\mathbb{R}}
\newcommand{\epsF}{\mathcal{F}_{d,n}^\varepsilon}
\newcommand{\F}{\mathcal{F}_{d,n}}

\title{Geometric Perspective on Concentration Phenomena in Frame Theory}
\author{
  Samuel A.\ Ballas \and
  Ferhat Karabatman \and
  Tom Needham
}
\date{\today}

\begin{document}

\maketitle

\begin{abstract}
Parseval and equal-norm frames play a fundamental role in frame theory and signal processing. It is known that a random frame, with unit vectors drawn independently from the uniform distribution on the sphere, will be nearly Parseval with high probability; asymptotic results go back at least to Goyal, Vetterli, and Thao and a non-asymptotic error bound was proved more recently by Kwok, Lau, and Ramachandran. In this work, we prove a dual result, which shows that random Parseval frames, with respect to the Haar measure, are nearly equal-norm with high probability. Our proofs are geometric in nature, and rely on general measure concentration principles in Riemannian manifolds. Using these techniques, we also give a novel probabilistic upper bound for the Paulsen problem.
\end{abstract}

\section{Introduction}\label{sec:intro}

Frames provide redundant systems of vectors for representing signals in Hilbert spaces. 
Unlike bases, frames allow stable reconstruction while permitting redundancy, a feature 
that makes them particularly useful in signal processing, data transmission, and related 
areas \cite{Christensen2016FramesRieszBases, CasazzaKutyniok2013, Casazza2000, KovacevicChebira2007}. In finite dimensions, a frame is simply a spanning family of vectors 
\((x_i)_{i=1}^n \subset \mathbb{R}^d\), and its redundancy is reflected in the fact that 
one may have \(n>d\). This redundancy is not merely a formal feature: it plays a central 
role in the stability of reconstruction, especially when the transmitted coefficients are 
affected by noise or erasures.

A central class of frames is given by  \emph{Parseval frames}. 
For such frames, the \emph{frame operator}, or the map $\R^d \to \R^d$ defined by 
\[
v \mapsto \sum_{i=1}^n \langle v, x_i \rangle x_i,
\]
is the identity map. The interpretation is that, in this case, the signal $v$ can be reconstructed from its sequence of linear measurements $(\langle v, x_i \rangle)_{i=1}^n \in \R^n$ in a particularly 
simple and stable manner. Equal-norm tight frames, where $\|x_i\|=\|x_j\|$ for all $i,j$, are also especially important because they 
distribute the importance of the linear measurements uniformly. This balance is closely connected 
to robustness under coefficient losses and to optimality properties in reconstruction 
problems. Several foundational results illustrate the importance of these structures. Goyal et 
al.\, showed that, under an additive noise model, the mean squared 
error is minimized precisely by tight frames~\cite{Goyal2001Quantized}. Subsequent works further emphasized the 
role of tight and equal-norm tight frames in reconstruction problems involving noise and 
erasures. For instance, Casazza and Kovačević~\cite{CasazzaKovacevic2003Erasures} and 
Holmes and Paulsen~\cite{HolmesPaulsen2004OptimalErasures} showed that equal-norm tight 
frames enjoy strong optimality properties in erasure settings. Taken together, these 
results indicate that both tightness and equal distribution of vector norms are essential 
features for stable and robust signal representation.

In light of the above, frames which are both Parseval and equal-norm are desirable in applications. Unfortunately, the space of equal-norm Parseval frames forms a complicated subvariety of $\R^{d \times n}$ which is hard to sample directly~\cite{BenedettoFickus2003,Dykema:2006ux,Cahill:2017gv,needham2022toric,mixon2024three}. On the other hand, spaces of frames which are only constrained to be Parseval or to be equal-norm (i.e., only satisfying one of the desired conditions) are easy to sample, as they can be identified with Stiefel manifolds or products of spheres, respectively. The goal of this paper is to address the sampling problem from a measure concentration perspective, addressing the natural question:  if one samples vectors satisfying one of the desired conditions (e.g., the Parseval condition), how likely is it that the resulting system is close to satisfying the other condition (e.g., the unit norm condition)? In other words, to what extent are 
the frame-theoretic structures that are optimal for reconstruction typical within natural 
random ensembles? 

\subsection{Contributions and Outline}

In this work, we address the questions above from the perspective of measure concentration on Riemannian manifolds (see~\cite{ledoux2001concentration,vershynin2020high}). An outline of the paper and descriptions of our main contributions are as follows:

\noindent \textbf{Concentration on Parseval Frames.}  In Section~\ref{concent_sphere_ball}, we consider random equal-norm frames, considered as i.i.d. samples from the uniform measure on a sphere. It is known in the frame theory community that random equal norm frames (i.e., sample from the sphere) are highly likely to be be nearly Parseval. Different versions of this statement are proved in~\cite{goyal1998quantized,bodmann2010road,kwok2019spectral}; see Section \ref{sec:related_work} for more details on related work. In  Theorem~\ref{concent_sphere}, we prove a version of this statement which originally appeared in~\cite{kwok2019spectral}: random unit norm frames are close to being Parseval with high probability, when the number of vectors in the frame is large---in fact, we arrived at this proof independently before being made aware of its existence in the literature, and the proof is included here for completeness and because it will be useful later in the paper. The proof relies on classical matrix concentration inequalities. Similar techniques allow us to establish concentration on Parseval frames for frames whose vector norms are not fully prescribed, but only bounded above (Theorem~\ref{concent_ball}). The tightness of these concentration bounds is explored via numerical experiments, included in the appendix.

\smallskip

\noindent \textbf{Concentration on Equal-Norm Frames.} We take the dual perspective in Section~\ref{sec:stiefel_manifold}, and consider concentration properties of random Parseval frames, considered as samples from a symmetric measure on the Stiefel manifold. We show in Theorem~\ref{thm:prevalence_epsH} that these frames are nearly equal-norm with high probability, assuming that the number of frame vectors is sufficiently large. The proof is more technical than the sphere case, and requires some estimates for geodesic distance and Ricci curvature of the Stiefel manifold. We also include numerical experiments illustrating this concentration result in the appendix.

\smallskip

\noindent \textbf{Application to the Paulsen Problem.} We give an application of our sphere concentration result to the Paulsen problem~\cite{bodmann2010road,casazza2013kadison,cahill2013paulsen,kwok2018paulsen,hamilton2021paulsen}. Roughly speaking, the Paulsen problem seeks a uniform bound on the squared distance from a frame that is \emph{$\varepsilon$-nearly equal-norm} and \emph{$\varepsilon$-nearly Parseval} (in a certain precise sense)—referred to briefly as an \emph{$\varepsilon$-nearly equal-norm Parseval frame}—to the subvariety of \emph{exactly} equal-norm Parseval frames. In Theorem~\ref{newpaulsen}, we derive an $O(\varepsilon^2 d)$ bound which holds with high probability for random nearly equal-norm and Parseval frames. Our bound is tighter than any of the deterministic bounds in the existing literature. In fact, we show in Proposition~\ref{prop:paulsen_bound} that any deterministic bound can be at best on the order of $O(\varepsilon d)$; this claim has been made in the literature, but to our knowledge no formal proof has appeared previously. This shows, in particular, that an optimal deterministic bound and our probabilistic bound are of different orders. Our proof is based on an extension of an example appearing in~\cite{ramachandran2021geodesic}. We note that similar probabilistic Parseval problem bounds, using a different sampling scheme, have appeared previously in~\cite{kwok2019spectral,ramachandran2021geodesic}; the distinction of our result is explained in the following subsection.

\subsection{Related Work}\label{sec:related_work}

Random frames have been studied in connection with quantization, erasures, uncertainty principles, and probabilistic constructions; see, e.g., \cite{goyal1998quantized,lyubarskii2010uncertainty,ehler2012random,vershynin2012introduction,bodmann2013random}. For instance, a basic model for random frames considers them as rectangular matrices with i.i.d. Gaussian entries. It follows from general measure concentration results that such a frame is highly likely to be nearly equal norm and nearly Parseval (after appropriate normalizations)---see \cite[Theorem 3.1.1 and 4.6.1]{vershynin2020high}, or related estimates in~\cite[Corollaries~3.2 and~3.5]{bodmann2013random}. More generally, Lyubarskii and Vershynin~\cite{lyubarskii2010uncertainty}, prove high-probability uncertainty principle estimates for random frame models, including subgaussian matrices, with applications to Kashin representations and vector quantization.

This paper is primarily concerned with random matrices sampled from distributions with additional structure imposed. There are several results in the literature which work with random frames whose vectors are sampled independently from a sphere, i.e., random unit-norm frames. Early work of Goyal, Vetterli, and Thao shows that the resulting frames are tight, asymptotically in the number of frame vectors~\cite[Theorem 1]{goyal1998quantized}. In the non-asymptotic regime, the earliest result we are aware of is due to Bodmann, who gives very general probabilistic guarantees on near-tightness for random fusion frames sampled from unitarily invariant distributions~\cite[Theorem 3.6]{bodmann2013random}; this specializes in the rank-1 case to essentially the setting of random unit norm frames. By working specifically with the case of unit norm frames, Kwok, Lau, and Ramachandran are able to derive much tighter probabilistic results~\cite{kwok2019spectral}. Theorem \ref{concent_sphere} below reproduces a result which can be found in the proof of~\cite[Lemma 5.4]{kwok2019spectral}, giving a non-asymptotic bound on the probability that a random unit norm frame is nearly Parseval. We include a proof of this result for the sake of completeness, and because Theorem \ref{concent_sphere} will be used in a crucial way to prove Theorem \ref{newpaulsen} later in the paper. Finally, we note that a stronger concentration result of this form is given in~\cite[Lemma A.13]{franks2020rigorous}, following from~\cite[Theorem 5.39]{vershynin2012introduction}, but these results involve unspecified constants. We prefer to work with fully explicit concentration bounds for the random models considered here, which allows one to examine the sharpness of these bounds through Monte Carlo simulations.

Rather than imposing the equal norm condition and studying the probability that the resulting frame is nearly Parseval, Theorem \ref{thm:prevalence_epsH} works in the opposite direction: we sample random Parseval frames by identifying the space of frames with the Stiefel manifold, endowed with Haar measure, and study the probability that the resulting frame is nearly equal norm. A related concentration phenomenon for random Parseval frames can also be inferred from the results of Hayden, Leung, and Winter~\cite{HaydenLeungWinter2006}. In particular, their concentration estimates for random projections imply, under the standard correspondence with Parseval frames, that the norm of each individual frame vector concentrates around the common equal-norm value. Our perspective is different: rather than controlling the deviation of the frame vectors from this value one at a time, we consider a frame chosen randomly from the space of Parseval frames and study its distance, as a whole, from the equal-norm condition. Thus, our result provides a global concentration statement for the frame rather than a vector-wise concentration estimate. This is in contrast to the more standard approach of deriving concentration results for a single frame then extending to the full frame via a union bound argument; see Remark \ref{remark:alternative_concentration_proof} for further discussion of the distinction.

The connection with the Paulsen problem was also considered in \cite{kwok2019spectral}. In that work, the authors obtain a high-probability \(O(\varepsilon^2 d)\)-type upper bound for the Paulsen distance in a random equal-norm model. Ramachandran later refined related results in his dissertation \cite{ramachandran2021geodesic}, using the result of \cite{franks2020rigorous}. These results are closely related to ours, but the underlying sampling models are different. The Paulsen problem concerns frames which are simultaneously \(\varepsilon\)-nearly equal-norm and \(\varepsilon\)-nearly Parseval. By contrast, the random frames considered in the works above are exactly equal-norm, rather than nearly equal-norm. Thus, the randomness is imposed on a smaller and more rigid class of frames than the one appearing naturally in the Paulsen problem. Our result addresses this distinction directly. We sample from the space of frames satisfying the two approximate Paulsen conditions themselves, and show that a high-probability \(O(\varepsilon^2 d)\) upper bound holds on this natural space. Moreover, our argument avoids the additional pseudorandomness and operator-scaling machinery used in some of the earlier works. Instead, we use a direct geometric approach: radial projection to the sphere, the explicit sphere concentration estimate from Theorem~\ref{concent_sphere}, the deterministic Hamilton--Moitra bound \cite{hamilton2021paulsen}, and the triangle inequality.

Ramachandran's dissertation \cite{ramachandran2021geodesic} also contains an \(O(\varepsilon d)\) upper bound for certain frames in a restricted range of parameters. This does not contradict the high-probability \(O(\varepsilon^2 d)\) estimate proved here. Rather, the two results apply in different regimes: the \(O(\varepsilon d)\) estimate is relevant when \(n\) is not too large relative to \(d\), whereas our \(O(\varepsilon^2 d)\) estimate holds with high probability once \(n\) is sufficiently large. To support the optimality of the \(O(\varepsilon d)\) behavior in the former regime, Ramachandran uses the \(d=2,n=4\) example appearing in Equation~(38) of \cite{casazza2012auto}.  In Proposition~\ref{prop:paulsen_bound}, we extend this construction to the regime \(n=2d\), showing that deterministic Paulsen bounds can be no better than order \(\Omega(\varepsilon d)\) in general.

\section*{Acknowledgements}

We would like to thank Akshay Ramachandran for providing detailed feedback earlier versions of the paper. TN was supported by NSF grants DMS 2324962 and CIF 2526630.

\section{Concentration of Parseval Frames}\label{concent_sphere_ball}

In this section, we establish concentration results for random vectors drawn from appropriately scaled spheres and balls. Our main results show that random frames drawn from these spaces are close to being Parseval with high probability, with respect to both the operator and Frobenius norms. We begin with a review of basic terminology and notation from frame theory.

\subsection{Frame Theoretic Concepts and Notation}

Throughout the paper, we consider $\R^d$ as an inner product space, endowed with its standard inner product $\langle \cdot, \cdot \rangle$ and norm $\|\cdot \|$. The adjoint (i.e., transpose) of a linear map $T:\R^d \to \R^n$ with respect to the standard inner products will be denoted by $T^\ast:\R^n \to \R^d$.

A finite sequence $(x_i)_{i=1}^n$ of vectors $x_i \in \R^d$ is called a \emph{frame} if it forms a spanning set for $\R^d$. Note that, for the sake of convenience, we consider a frame to be an ordered tuple of vectors (rather than an unordered set, which is an alternative convention).  

For any finite sequence of vectors $(x_i)_{i=1}^n$, we have the associated \emph{analysis operator} $T:\R^d \to \R^n$, defined by 
$T(x)=(\langle x,x_1\rangle,\dots,\langle x,x_n\rangle)$, and its adjoint, the \emph{synthesis operator} $T^\ast :\R^n \to \R^d$, $T^\ast (c_1,\dots,c_n)=\sum_{i=1}^n c_i x_i$. The \emph{frame operator} is $S =  T^*T: \R^d \to \R^d$, which we write as a sum of outer products $S = \sum_{i=1}^n x_i x_i^\ast$. In the context of signal processing, where $x \in \R^d$ is considered as a signal, application of the analysis operator amounts to taking linear measurements encoded by the sequence, and application of the synthesis operator amounts to reconstructing a signal from these measurements. A simple linear algebra exercise shows that a sequence is a frame if and only if its frame operator $S$ is invertible.

For frames in finite-dimensional spaces, there are certain natural classes of frames which are of particular interest. First, it is especially convenient when the frame operator is not just invertible, but is actually the identity map; in this case, the frame is said to be \emph{Parseval}. Second, one may wish for the vectors in the frame to be \emph{equal-norm}. If a frame $(x_i)_{i=1}^n$ in $\R^d$ is both Parseval and equal-norm, a simple calculation shows that the norm must be given by $\|x_i\|^2 = \frac{d}{n}$. 

In short, the results of this paper address the following question: if one of the desired properties above is enforced (e.g., the equal-norm condition), how likely is it that a random frame is close to satisfying the other property (e.g., Parseval)?

\subsection{Concentration for Spheres}

We begin by considering concentration properties for equal-norm frames. We use the notation $\mathbb{S}^{d-1}_R$ for the $(d-1)$-dimensional sphere of radius $R > 0$, considered as a submanifold of $\R^{d}$. We endow the sphere with the Riemannian metric induced from $\R^{d}$, and consider it as a measure space with normalized Riemannian volume; i.e., with the uniform probability measure. 

We consider concentration properties framed in terms of operator norms. For a linear operator on Euclidean space $A:\R^d \to \R^d$, we use $\|A\|_\mathrm{op}$ to denote its operator norm and $\|A\|_\mathrm{F}$ to denote its Frobenius norm. The main result of this section is the following.

\begin{theorem}[Concentration over the Sphere~\cite{kwok2019spectral}]\label{concent_sphere}
Let $(x_i)_{i=1}^n$ be a set of independent random vectors, drawn from the uniform distribution on $\mathbb{S}^{d-1}_{\sqrt{d/n}}$, and let  $S$ be the associated frame operator. Then, for every $\varepsilon \ge 0$,
\[
\mathbb{P}\!\left( \| S - I_d \|_{\mathrm{op}} \ge \varepsilon \right)
\;\le\;
2d \exp\!\left(
-\frac{n \varepsilon^2}
{\,2d - 2 + \tfrac{2\varepsilon(d-1)}{3}\,}
\right)
\]
and
\[
\mathbb P\!\left(\|S-I_d\|_{\mathrm F}\ge \varepsilon\right)
\le
2d \exp\!\left(
-\frac{n \varepsilon^2}
{\,2d^2 - 2d + \tfrac{2\varepsilon \sqrt{d}(d-1)}{3}\,}
\right).
\]
\end{theorem}

In the language of frame theory, this result has the following interpretation: for fixed $d$, a random equal-norm frame consisting of $n$ vectors is arbitrarily close to being Parseval with overwhelmingly high probability, provided $n$ is sufficiently large. This is concretely quantified in Corollary~\ref{cor:confidence_bounds_sphere_ball} below. As was explained in the introduction, this result appears in the proof of \cite[Lemma 5.4]{kwok2018paulsen}. Before being made aware of its existence in the literature, we had independently found a proof, using the same main tool (the Matrix Berstein Inequality, presented below). We provide a proof of the theorem for the sake of completeness, and because the result will be useful later on in the proof of Theorem \ref{newpaulsen}.

\begin{theorem}[Matrix Bernstein Inequality {\cite[Theorem 1.4]{tropp2012user}}]\label{bernstein}
Let $Y$ be a random self-adjoint $d \times d$ matrix such that $\mathbb{E}[Y]=0$ and $\|Y\|_{\mathrm{op}} \le A$ almost surely, for some $A>0$. Let $Y_1,\dots,Y_n$ be i.i.d.\ copies of $Y$. Then, for all $\varepsilon \ge 0$,
\[
\mathbb{P}\left(\left\|\sum_{i=1}^n Y_i\right\|_{\mathrm{op}} \ge \varepsilon \right)
\le 2d \cdot \exp\!\left( \frac{-{\varepsilon}^2/2}{\sigma^2 + A\varepsilon/3} \right),
\qquad
\text{where}\quad
\sigma^2 \coloneqq \left\|\sum_{i=1}^n \mathbb{E}[Y_i^2]\right\|_{\mathrm{op}}.
\]
\end{theorem}

\begin{proof}[Proof of Theorem~\ref{concent_sphere}]
Let $v = (v^{(1)},\ldots,v^{(d)})$ be a random vector drawn uniformly from the sphere $\mathbb{S}^{d-1}_{\sqrt{d/n}}$. By rotational invariance,
\[
\mathbb{E}[v^{(i)} v^{(j)}] = 0 \quad \text{for } i \neq j,
\qquad
\mathbb{E}[(v^{(i)})^2] = \frac{1}{n}, \qquad \text{and hence} \hspace{0.2cm} \mathbb{E}[v v^*] = \frac{1}{n} I_d.
\]

Let $(x_i)_{i=1}^n$ be independent random vectors with the same distribution as $v$, and define the frame operator \(S = \sum_{i=1}^n x_i x_i^*\). Then
\[
\mathbb{E}[S] = \sum_{i=1}^n \mathbb{E}[x_i x_i^*] = n \cdot \frac{1}{n} I_d = I_d.
\]

Define \(Y_i := x_i x_i^* - \frac{1}{n} I_d\). Then $\mathbb{E}[Y_i]=0$. Since $\|x_i\|^2=d/n$, the matrix $x_i x_i^*$ has eigenvalues $d/n$ and $0$ (with multiplicity $d-1$). Hence $Y_i$ has eigenvalues $\frac{d-1}{n}$ and $-\frac1n$ (with multiplicity $d-1$), so for $d\ge2$,
\[
\|Y_i\|_{\mathrm{op}}=\frac{d-1}{n}.
\]
Moreover, a direct computation shows that 
\[
\left\|\sum_{i=1}^n \mathbb{E}[Y_i^2]\right\|_{\mathrm{op}} = \frac{d-1}{n}.
\]
Applying Theorem~\ref{bernstein} yields the desired bound for the operator norm.

\medskip

The result in terms of the Frobenius norm follows directly from the relationship between the operator norm and the Frobenius norm:
\[
\sqrt{d}\,\|A\|_{\mathrm{op}} \geq \|A\|_F.
\]
\end{proof}

\begin{remark}
The same approach extends beyond the uniform sphere measure. In general, one may define
\[
Y_i := x_i x_i^* - \mathbb{E}[x_i x_i^*],
\]
so that \(\mathbb{E}[Y_i] = 0\). Thus, the concentration result in Theorem~\ref{bernstein} applies as long as the random vectors \(x_i\) satisfy $\mathbb{E}[x_i x_i^*] = \alpha I_d$ for some \(\alpha > 0\). In particular, this includes distributions supported on the sphere or the ball that are isotropic after appropriate normalization.
\end{remark}

\subsection{Concentration for Euclidean Balls}

The next main result is a concentration result for frames whose vectors are drawn from Euclidean balls. We use $\mathbb{B}^d_R$ to denote the closed Euclidean ball in $\R^d$ of radius $R > 0$. This is treated as a metric measure space with Euclidean metric and uniform probability measure. The main result of this subsection  is as follows.

\begin{theorem}[Concentration in the Ball]\label{concent_ball}

Let $(x_i)_{i=1}^n$ be a set of independent random vectors, drawn from the uniform distribution on  $\mathbb{B}_{\sqrt{\frac{d+2}{n}}}^d$ and let  $S$ be the associated frame operator. Then, for every $\varepsilon \ge 0$,
\[
\mathbb{P}\!\left(
\left\|S - I_d\right\|_{\mathrm{op}} \ge \varepsilon
\right)
\;\le\;
2d \exp\!\left(-\,\frac{n \varepsilon^2}{2d+2+\frac{2 \varepsilon(d+1)}{3}}\right)
\]
and
\[
\mathbb{P}\!\left( \| S - I_d \|_{\mathrm{F}} \ge \varepsilon \right)
\;\le\;
2d \exp\!\left(-\,\frac{n \varepsilon^2}{2d^2+2d+\frac{2 \varepsilon\sqrt{d}(d+1)}{3}}\right)
\].
\end{theorem}

Sharper estimates could be derived from the methods of \cite{franks2020rigorous}, but these are largely based on~\cite[Theorem 5.39]{vershynin2012introduction}, which involves unspecified constants. The proof strategy presented below yields the explicit expressions in the theorem.

\begin{lemma}\label{ball_radius}
Let $v$ be a random vector uniformly distributed on the Euclidean ball $\mathbb{B}_r^d$.
Then
\[
\mathbb{E}[v v^*] = \frac{r^2}{d+2}\, I_d.
\]
In particular, choosing $r = \sqrt{\frac{d+2}{n}}$ ensures that $\mathbb{E}[v v^*] = \frac{1}{n}I_d$.
\end{lemma}

\begin{proof}
Since the uniform distribution on $\mathbb{B}_r^d$ is rotationally invariant, the matrix
$\mathbb{E}[v v^*]$ must be a scalar multiple of the identity.
In particular, all diagonal entries are equal, and we have $\mathbb{E}[v_i^2] = \frac{1}{d}\,\mathbb{E}[\|v\|^2]$. Let
\[
k = \frac{1}{d}\operatorname{tr}\bigl(\mathbb{E}[v v^*]\bigr)
= \frac{1}{d}\,\mathbb{E}[\|v\|^2].
\]
Then
\[
\mathbb{E}[v v^*] = k I_d.
\]

To compute $k$, note that
\[
k
= \frac{1}{d\,\operatorname{Vol}(\mathbb{B}_r^d)} \int_{\mathbb{B}_r^d} \|v\|^2 \, dV.
\]
Using polar coordinates $(\rho,\theta)$, where $\|v\| = \rho$ and
$dV = \rho^{d-1}\, d\rho \, d\theta$, we obtain
\[
\int_{\mathbb{B}_r^d} \|v\|^2 \, dV
= \int_0^{r} \int_{S^{d-1}} \rho^2 \rho^{d-1} \, d\theta \, d\rho
= \left( \int_0^{r} \rho^{d+1}\, d\rho \right)\operatorname{Vol}(S^{d-1})
= \frac{r^{d+2}}{d+2}\operatorname{Vol}(S^{d-1}).
\]
On the other hand, the volume of $\mathbb{B}_r^d$ is given by
\[
\operatorname{Vol}(\mathbb{B}_r^d)
= \int_0^{r} \rho^{d-1}\, d\rho \cdot \operatorname{Vol}(S^{d-1})
= \frac{r^d}{d}\operatorname{Vol}(S^{d-1}).
\]
Substituting these expressions yields
\[
k
= \frac{\frac{r^{d+2}}{d+2}\operatorname{Vol}(S^{d-1})}
{d \cdot \frac{r^d}{d}\operatorname{Vol}(S^{d-1})}
= \frac{r^2}{d+2}.
\]
Consequently,
\[
\mathbb{E}[v v^*] = \frac{r^2}{d+2}\, I_d.
\]
Choosing $r = \sqrt{\frac{d+2}{n}}$ makes the expected value $\frac{1}{n}I_d$.
\end{proof}

\begin{proof}[Proof of Theorem~\ref{concent_ball}]
For the concentration result in the operator norm, we proceed exactly as in the corresponding part of the proof of Theorem~\ref{concent_sphere}. By Lemma~\ref{ball_radius}, the radius of the ball must be \(\sqrt{(d+2)/n}\) in order to ensure that \(\mathbb{E}[S] = I_d\). For the Frobenius norm, we proceed in the same way as in Theorem~\ref{concent_sphere}.
\end{proof}

We conclude this section with a corollary that quantifies our concentration in terms of increasing $n$ (the number of vectors in the frames). The proof follows immediately from the main theorems.

\begin{corollary}\label{cor:confidence_bounds_sphere_ball}
Let \(\delta \in (0,1)\) and \(\varepsilon > 0\). Then, for both the sphere and ball models, there exists a constant $C_{d,\varepsilon}$ (depending on dimension $d$ and tolerance $\varepsilon$) such that if
\[
n \;\ge\; C_{d,\varepsilon}\,\log\!\left(\frac{2d}{\delta}\right),
\]
then
\[
\mathbb{P}\bigl(\|S-I_d\|_{\mathrm{op}} \ge \varepsilon\bigr)\le \delta.
\]

In particular:
\begin{itemize}
    \item To guarantee a confidence level of \(95\%\), it suffices to take
    \[
    n \;\ge\; C_{d,\varepsilon}\,\log(40d).
    \]
    \item To guarantee a confidence level of \(99\%\), it suffices to take
    \[
    n \;\ge\; C_{d,\varepsilon}\,\log(200d).
    \]
\end{itemize}

Here one may take
\[
C_{d,\varepsilon}=
\frac{2d-2+\frac{2\varepsilon(d-1)}{3}}{\varepsilon^2}
\]
for the sphere model, and
\[
C_{d,\varepsilon}=
\frac{2d+2+\frac{2\varepsilon(d+1)}{3}}{\varepsilon^2}
\]
for the ball model.

An analogous statement holds for the Frobenius norm.
\end{corollary}

\section{Concentration of Equal-Norm Frames}\label{sec:stiefel_manifold}

In the previous section, we showed that random collections of equal-norm vectors concentrate around Parseval frames. In this section, we prove a complementary result: random Parseval frames concentrate around equal-norm frames. More precisely, we show that if a Parseval frame is chosen at random, then the norms of its frame vectors are close to the common norm of an equal-norm Parseval frame, namely \(\sqrt{d/n}\), with high probability.

A Parseval frame \((x_i)_{i=1}^n \subset \mathbb{R}^d\) can be identified with a matrix \(U \in \mathbb{R}^{n\times d}\) whose rows are the vectors \(x_i\). Such a matrix thus satisfies the condition \(U^*U=I_d\), so that  Parseval frames correspond to the points of the \emph{Stiefel manifold} 
\begin{equation}\label{eqn:stiefel_manifold}
\mathrm{St}(n,d) \coloneqq \{U \in \R^{n \times d} \mid U^\ast U = I_d\}.
\end{equation}
We note that the correspondence between Parseval frames and the Stiefel manifold has appeared previously in the frame theory literature, e.g.,~\cite{Dykema:2006ux,needham2021symplectic}.  Throughout this section, we equip \(\mathrm{St}(n,d)\) with its canonical Riemannian metric (described in the following subsection) and the associated normalized Riemannian volume measure, which we refer to as the \emph{symmetric measure}.

We can now state the main result of this section.

\begin{theorem}[Concentration over the Stiefel Manifold]
\label{thm:prevalence_epsH}
Let $(x_i)_{i=1}^n$ be a random Parseval frame in $\R^d$, drawn from the symmetric measure. Let $\varepsilon > 0$, assume that $n>d\geq 2$, and that $n$ is sufficiently large so that
\[
\frac{3\sqrt{17}}{4}
\sqrt{\frac{d+\ln n}{n}}
< \varepsilon.
\]
Then we have
\[
\mathbb{P}\!\left(
\exists\, i \in \{1,\dots,n\}
\;\text{such that}\;
\bigl|\,\|x_i\|-\sqrt{d/n}\,\bigr|
\ge \varepsilon
\right)
\le
\exp\!\left(
-\frac{(n-2)}{8}
\left(
\varepsilon-
\frac{3\sqrt{17}}{4}
\sqrt{\frac{d+\ln n}{n}}
\right)^2
\right).
\]
\end{theorem}

\begin{remark}
    The theorem shows that random Parseval frames concentrate on frames whose norms are all $\varepsilon$-close to $\sqrt{d/n}$, in an additive sense. Frequently in frame theory, an alternative notion of $\varepsilon$-nearly equal-norm frames, based on relative error, is considered. This notion is considered in Section~\ref{sec:paulsen_problem}, in relation to the Paulsen problem.
\end{remark}

In analogy with Corollary \ref{cor:confidence_bounds_sphere_ball}, the theorem immediately yields a quantification of concentration in $n$.

\begin{corollary}\label{confidence}
Fix $\varepsilon > 0$. Under the assumptions of 
Theorem~\ref{thm:prevalence_epsH}, suppose further that
\[
\varepsilon
\geq
\frac{3\sqrt{17}}{2}
\sqrt{\frac{d+\ln n}{n}}.
\]
If
\[
n \ge 2+\frac{32\ln(20)}{\varepsilon^2},
\]
then the conclusion of Theorem~\ref{thm:prevalence_epsH} holds with
confidence at least \(95\%\). If
\[
n \ge 2+\frac{32\ln(100)}{\varepsilon^2},
\]
then the conclusion holds with confidence at least \(99\%\).
\end{corollary}

The proof of Theorem~\ref{thm:prevalence_epsH} is geometric. It relies on a concentration inequality on manifolds together with two key ingredients: an upper bound for the expectation of a suitable Lipschitz function, and a lower bound on the Ricci curvature of \(\mathrm{St}(n,d)\).

\begin{theorem}[\cite{ledoux2001concentration}]\label{ricci}
Let \(M\) be a compact, connected Riemannian manifold with \(\mathrm{Ric}\ge \kappa>0\), and let \(f:M\to\mathbb{R}\) be Lipschitz with constant \(\mathrm{Lip}(f)\). If \(X\) is distributed according to the normalized Riemannian volume measure on \(M\), then for all \(t>0\),
\[
\mathbb{P}\!\left(f(X)-\mathbb{E}[f(X)]\ge t\right)
\le
\exp\!\left(-\frac{\kappa t^2}{2\,\mathrm{Lip}(f)^2}\right).
\]
\end{theorem}

The rest of this section is devoted to the proof of Theorem \ref{thm:prevalence_epsH}, which is more technical than the proofs of the concentration results presented in the previous section. We note that a concentration result can be derived somewhat straightforwardly from Theorem \ref{ricci} together with a union bound. However, the proof of Theorem \ref{thm:prevalence_epsH} avoids the use of a union bound and thereby yields a stronger concentration estimate; see Remark \ref{remark:alternative_concentration_proof} for details.

\subsection{Geometric Ingredients}

Since the main result follows directly from Theorem~\ref{ricci}, we need two geometric ingredients: the geodesic distance and the Ricci curvature for the Stiefel manifold. For the geodesic distance, it is enough for our purposes to understand its relationship with the Frobenius norm. Indeed, we will define a function whose Lipschitz constant can be estimated easily with respect to the Frobenius norm, and the comparison between the Frobenius norm and the geodesic distance will then allow us to obtain the required Lipschitz constant. For the Ricci curvature, it will be sufficient to establish an appropriate lower bound.

Let us recall some basic geometric structures of the Stiefel manifold; a  standard reference for details is~\cite{edelman1998geometry}.  The tangent space at \( U \in \mathrm{St}(n, d) \) is given by
\[
T_U \mathrm{St}(n, d) = \left\{ \Delta \in \mathbb{R}^{n \times d} \;\middle|\; U^* \Delta + \Delta^* U = 0 \right\}.
\]
Equivalently, any tangent vector \( \Delta \in T_U \mathrm{St}(n, d) \) can be expressed in the form
\[
\Delta = UA + U_\perp B, 
\]
where $A$ is a skew-symmetric $d \times d$ matrix, $B \in \R^{(n-d) \times d}$ is an arbitrary matrix, and  \( U_\perp \in \mathbb{R}^{n \times (n - d)} \) is an orthonormal basis for the complement of the span of the columns of \( U \). The canonical Riemannian metric is defined  at \( U \in \mathrm{St}(n, d) \) as the map
\begin{align*}
    T_U \mathrm{St}(n,d) \times T_U \mathrm{St}(n,d) &\to \R \\
    (\Delta_1, \Delta_2) &\mapsto \mathrm{Tr} \left( \Delta_1^* \left( I_n - \frac{1}{2} U U^* \right) \Delta_2 \right).
\end{align*}
By general principles, this Riemannian metric yields a volume measure, which we normalize to define the \emph{symmetric measure}, used in the statement of Theorem \ref{thm:prevalence_epsH}.
We use \(d_C\) to denote the geodesic distance with respect to the canonical metric. As stated above, to deduce that  certain functions are Lipschitz with respect to $d_C$, the following bound in terms of Frobenius distance will suffice.

\begin{lemma}\label{lem:geodesic-frobenius}
For \(U_1,U_2\in \mathrm{St}(n,d)\),
\[
d_C(U_1,U_2)\ge \frac{1}{\sqrt2}\,\|U_1-U_2\|_{\mathrm F}.
\]
\end{lemma}

\begin{proof}
By~\cite{mataigne2026bounds},
\[
d_C(U_1,U_2)
\ge
\sqrt{\tfrac12}\cdot 2\sqrt d \,
\arcsin\!\left(\frac{\|U_1-U_2\|_{\mathrm F}}{2\sqrt d}\right).
\]
The result follows, since \(\arcsin(x)\ge x\) for \(x\in[0,1]\).
\end{proof}

Next, we derive a lower bound on Ricci curvature.

\begin{lemma}\label{lem:ricci-stiefel}
Let \(\xi \in T_U\mathrm{St}(n,d)\) be a unit tangent vector with respect to the canonical metric. If \(n>d+1\), then
\[
\mathrm{Ric}(\xi,\xi)\ge \frac n2-1.
\]
\end{lemma}

\begin{proof}
Write \(\xi=UA+U_\perp B\), where \(A\) is skew-symmetric and \(B\in \mathbb{R}^{(n-d)\times d}\). The explicit formula for the Ricci curvature under the canonical metric, obtained from~\cite{nguyen2022curvatures}, gives
\[
\mathrm{Ric}(\xi,\xi)
=
\left(\frac{2-d}{4}+\frac{d-n}{4}\right)\mathrm{Tr}(A^2)
+
\left(\frac{1-d}{2}+(n-2)\right)\mathrm{Tr}(B^*B).
\]
Since \(\mathrm{Tr}(A^2)=-\|A\|^2\), this becomes
\[
\mathrm{Ric}(\xi,\xi)
=
\frac{n-2}{4}\|A\|^2
+
\left(n-\frac d2-\frac32\right)\|B\|^2.
\]
Because \(\xi\) is a unit vector with respect to the canonical metric,
\[
\frac12\|A\|^2+\|B\|^2=1.
\]
Optimizing under this constraint yields
\[
\mathrm{Ric}(\xi,\xi)\ge
\min\!\left\{
\frac n2-1,\,
n-\frac d2-\frac32
\right\}.
\]
The result follows if $n > d+1$.
\end{proof}

\subsection{Analytic Ingredients}

We now introduce the function that will be used in the proof of Theorem~\ref{thm:prevalence_epsH}. We will estimate its Lipschitz constant and compute its expected value. For \(G\in \mathbb{R}^{n\times d}\), define
\[
f(G):=\max_{1\le i\le n}\|g_i\|,
\]
where \(g_i\in \mathbb{R}^d\) denotes the \(i\)-th row of \(G\). When restricted to \(\mathrm{St}(n,d)\), this function records the maximum row norm of the corresponding Parseval frame.

We first estimate \(\mathbb{E}[f(U)]\), where \(U\) is drawn from the symmetric measure on \(\mathrm{St}(n,d)\). Since the row norms of \(U\) are not independent, we compare \(U\) to a Gaussian matrix using the following result of Tropp.

\begin{theorem}[{\cite[Theorem~1]{tropp2012comparison}}]\label{sublinear}
Let \(U\) be a random element of \(\mathrm{St}(n,d)\) drawn from the symmetric measure, and let \(G\in\mathbb{R}^{n\times d}\) be a random matrix with independent \(\mathcal{N}(0,n^{-1})\) entries. For each nonnegative, sublinear, convex function \(f\) on \(\mathbb{R}^{n\times d}\) and each weakly increasing, convex function \(\Phi:\mathbb{R}\to\mathbb{R}\), we have
\[
\mathbb{E}\,\Phi\!\bigl(f(U)\bigr)
\le
\mathbb{E}\,\Phi\!\left(\left(1+\frac{d}{2n}\right)f(G)\right).
\]
\end{theorem}

We are now ready to bound the expected value.

\begin{proposition}\label{prop:EfU}
Assume \(n>d\geq 2\). Then
\[
\mathbb{E}[f(U)]
\le
\frac{3\sqrt{17}}{4}
\sqrt{\frac{d+\ln n}{n}}.
\]
\end{proposition}

\begin{proof}
Take \(\Phi\) to be the identity function in Theorem~\ref{sublinear}, which is weakly increasing and convex. The function \(f\) is nonnegative, convex, and sublinear on \(\mathbb{R}^{n\times d}\), hence
\[
\mathbb{E}[f(U)]
\le
\left(1+\frac{d}{2n}\right)\mathbb{E}[f(G)].
\]

We now compute the expected value of \(f(G)\). This follows from a straightforward calculation, since the vectors are independent and the distribution of each vector can be identified explicitly. \(\|g_i\|^2\sim \frac1n\chi_d^2\) since each component of \(g_i\) has \(\mathcal{N}(0,\frac{1}{n})\).

Since \(f(G)\ge0\), we use the standard identity
\[
\mathbb{E}[f(G)] = \int_0^\infty \mathbb{P}(f(G)>t)\,dt.
\]
Moreover, by independence,
\[
\mathbb{P}(f(G)>t)
=
1-\left(\mathbb{P}(\|g_i\|\le t)\right)^n
=
1-\left(\mathbb{P}\left(\frac 1n \chi_d^2\le t^2\right)\right)^n
=
1-\left(\mathbb{P}(\chi_d^2\le nt^2)\right)^n.
\]
Using \(1-x^n\le n(1-x)\) for \(x\in[0,1]\), we obtain
\[
\mathbb{P}(f(G)>t)\le n\,\mathbb{P}(\chi_d^2\ge nt^2).
\]
Hence, for every \(A>0\),
\begin{align*}
\mathbb{E}[f(G)] &= \int_0^A \mathbb{P}(f(G)>t)\,dt + \int_A^\infty \mathbb{P}(f(G)>t)\,dt \\ 
&\le
\int_0^A 1\,dt +
\int_A^\infty n\,\mathbb{P}(\chi_d^2\ge nt^2)\,dt\\
&=
A+n\int_A^\infty \mathbb{P}(\chi_d^2\ge nt^2)\,dt.
\end{align*}

Now apply Laurent--Massart's inequality~\cite[Lemma~1]{laurent2000adaptive}: if \(X\sim \chi_d^2\), then for all \(x\ge0\),
\[
\mathbb{P}\!\left(X-d\ge 2\sqrt{dx}+2x\right)\le e^{-x}.
\]
Solving $2\sqrt{dx}+2x = nt^2-d$ for \(x\) yields
\[
x=\frac{\left(\sqrt{2nt^2-d}-\sqrt d\right)^2}{4},
\]
and therefore
\[
\mathbb{P}(\chi_d^2\ge nt^2)
\le
\exp\!\left(
-\frac{\left(\sqrt{2nt^2-d}-\sqrt d\right)^2}{4}
\right).
\]
Substituting this into the previous inequality and performing the change of variables
\[
r=\frac{\sqrt{2nt^2-d}-\sqrt d}{2}
\]
gives
\[
\mathbb{E}[f(G)]
\le
A+\int_B^\infty
\frac{2r+\sqrt d}{\sqrt{\frac{(2r+\sqrt d)^2+d}{2n}}}
\,e^{-r^2}\,dr, \qquad \text{where} \hspace{0.2cm}
A=\sqrt{\frac{(2B+\sqrt d)^2+d}{2n}}.
\]
Bounding the prefactor by \(\sqrt{2n}\), we obtain
\[
\mathbb{E}[f(G)]
\le
A+\sqrt{2n}\int_B^\infty e^{-r^2}\,dr.
\]
Using
\[
\int_B^\infty e^{-r^2}\,dr \le \frac{e^{-B^2}}{2B},
\qquad B>0,
\]
we conclude that
\[
\mathbb{E}[f(G)]
\le
\sqrt{\frac{(2B+\sqrt d)^2+d}{2n}}
+\sqrt{2n}\,\frac{e^{-B^2}}{2B}.
\]
Choosing \(B=\sqrt{\ln n}\) yields
\[
\mathbb{E}[f(G)]
\le
\sqrt{\frac{4\ln n+4\sqrt{d\ln n}+2d}{2n}}
+\frac{1}{\sqrt{2n\ln n}}.
\]
Since \(n>d\geq 2\), we have
\begin{align*}
\mathbb{E}[f(U)]
&\leq
\left(1+\frac{d}{2n}\right)
\left(
\sqrt{\frac{4\ln n+4\sqrt{d\ln n}+2d}{2n}}
+
\frac{1}{\sqrt{2n\ln n}}
\right) \\
&=
\frac{1}{\sqrt{n}}
\left(1+\frac{d}{2n}\right)
\left(
\sqrt{d+2\ln n+2\sqrt{d\ln n}}
+
\frac{1}{\sqrt{2\ln n}}
\right) \\
&\leq
\frac{1}{\sqrt{n}}
\left(1+\frac{d}{2n}\right)
\left(
\sqrt{d}+\sqrt{2\ln n}
+
\frac{1}{\sqrt{2\ln n}}
\right) \\
&\leq
\frac{3}{2\sqrt{n}}
\left(
\frac{3}{2}\sqrt{d}+\sqrt{2\ln n}
\right) \\
&\leq
\frac{3\sqrt{17}}{4}
\sqrt{\frac{d+\ln n}{n}}.
\end{align*}
This gives the claimed result.
\end{proof}

We next compute a Lipschitz constant for \(f\) on \(\mathrm{St}(n,d)\).

\begin{proposition}\label{prop:lipschitz-f}
The restriction of \(f\) to \(\mathrm{St}(n,d)\) is \(\sqrt2\)-Lipschitz with respect to the geodesic distance \(d_C\).
\end{proposition}

\begin{proof}
Let \(U,V\in \mathrm{St}(n,d)\), and let \(u_i\) and \(v_i\) denote their \(i\)-th rows. Then
\[
|f(U)-f(V)|
=
\left|\max_i\|u_i\|-\max_i\|v_i\|\right|
\le
\max_i\left|\|u_i\|-\|v_i\|\right|
\le
\max_i\|u_i-v_i\|.
\]
Therefore,
\[
|f(U)-f(V)|
\le
\bigl\|(\|u_1-v_1\|,\dots,\|u_n-v_n\|)\bigr\|_\infty
\le
\bigl\|(\|u_1-v_1\|,\dots,\|u_n-v_n\|)\bigr\|_2
=
\|U-V\|_{\mathrm F}.
\]
Thus \(f\) is \(1\)-Lipschitz with respect to the Frobenius norm. Applying Lemma~\ref{lem:geodesic-frobenius} yields the result.
\end{proof}

\subsection{Proof of Theorem~\ref{thm:prevalence_epsH}}

We are now prepared to prove the main theorem of this section.
\begin{proof}[Proof of Theorem~\ref{thm:prevalence_epsH}]
Let \(U\in \mathrm{St}(n,d)\) be the matrix associated with the random Parseval frame, and let \(u_i\) denote its rows. Since
\[
\frac{3\sqrt{17}}{4}
\sqrt{\frac{d+\ln n}{n}}
<\varepsilon,
\]
we have \(\varepsilon>\sqrt{\frac{d}{n}}\). Therefore,
\(\left|\|u_i\|-\sqrt{\frac dn}\right|\ge \varepsilon
\)
can only occur when
\(\|u_i\|-\sqrt{\frac dn}\ge \varepsilon\),
because the alternative inequality
\(\sqrt{\frac dn}-\|u_i\|\ge \varepsilon\) would force
\(\|u_i\|\le \sqrt{\frac dn}-\varepsilon<0\), which is impossible. Hence
\[
\mathbb{P}\!\left(
\exists\, i \text{ such that }
\left|\|u_i\|-\sqrt{\frac dn}\right|\ge \varepsilon
\right)
=
\mathbb{P}\!\left(
\max_{1\le i\le n}\|u_i\|-\sqrt{\frac dn}\ge \varepsilon
\right).
\]
Since \(f(U)=\max_i\|u_i\|\), this equals
\[
\mathbb{P}\!\left(
f(U)-\sqrt{\frac dn}\ge \varepsilon
\right)
=
\mathbb{P}\!\left(
f(U)-\mathbb{E}[f(U)]
\ge
\varepsilon-\bigl(\mathbb{E}[f(U)]-\sqrt{\tfrac dn}\bigr)
\right).
\]
By Proposition~\ref{prop:EfU},
\[
\mathbb{E}[f(U)]
\le
\frac{3\sqrt{17}}{4}
\sqrt{\frac{d+\ln n}{n}},
\]
and therefore
\[
\mathbb{P}\!\left(
\exists\, i \text{ such that }
\left|\|u_i\|-\sqrt{\frac dn}\right|\ge \varepsilon
\right)
\le
\mathbb{P}\!\left(
f(U)-\mathbb{E}[f(U)]
\ge
\varepsilon-
\frac{3\sqrt{17}}{4}
\sqrt{\frac{d+\ln n}{n}}
\right).
\]
By Proposition~\ref{prop:lipschitz-f}, \(f\) is \(\sqrt2\)-Lipschitz. Moreover, by Lemma~\ref{lem:ricci-stiefel}, the Ricci curvature of \(\mathrm{St}(n,d)\) under the canonical metric is bounded below by \(\frac{n}{2}-1\). Applying Theorem~\ref{ricci} with
\[
\kappa=\frac n2-1
\qquad \text{and} \qquad
\mathrm{Lip}(f)=\sqrt2
\]
gives
\[
\mathbb{P}\!\left(f(U)-\mathbb{E}[f(U)]\ge t\right)
\le
\exp\!\left(-\frac{(n-2)t^2}{8}\right),
\qquad t\ge0.
\]
Finally, substituting
\[
t=
\varepsilon-
\frac{3\sqrt{17}}{4}
\sqrt{\frac{d+\ln n}{n}},
\]
which is positive by assumption, yields the claim.
\end{proof}

In the proof above, we employed an explicit upper bound on
\(\mathbb{E}[f(G)]\). In principle, one could obtain a sharper estimate
by using the exact expression for \(\mathbb{E}[f(G)]\), which involves
the incomplete gamma function and the gamma function. Since our goal
was to state a clean explicit concentration bound and keep the resulting
inequality as simple as possible, we used an upper bound for
\(\mathbb{E}[f(G)]\). For completeness, we record below the concentration
result obtained when the exact expression is used.

\begin{proposition}\label{main}
Let $(x_i)_{i=1}^n$ be a random Parseval frame in $\R^d$, drawn from the symmetric measure, and let $\varepsilon>0$ be given. Then, for all sufficiently large $n$ satisfying $K<\varepsilon$, we have
\[
\mathbb{P}\!\left(
\max_{i \in \{1,\dots,n\}} \|x_i\| - \sqrt{\frac{d}{n}} \ge \varepsilon
\right)
\le
\exp\!\left(
-\frac{(n-2)(\varepsilon-K)^2}{8}
\right),
\]
where
\[
K
\coloneqq
\left(1+\frac{d}{2n}\right)\frac{1}{\sqrt{2n}}
\int_0^\infty
\frac{
1 -
\left(
\dfrac{\gamma\!\left(\frac{d}{2},y\right)}
{\Gamma\!\left(\frac{d}{2}\right)}
\right)^n
}{\sqrt{y}}
\,dy
-
\sqrt{\frac{d}{n}}.
\]
\end{proposition}

Although this formula provides a sharper result, we chose to state the previous version instead. The reason is that the constant \(K\) appearing here is rather complicated, whereas in the previous formulation we replace it by a simpler upper bound that can be stated more transparently. If \(n \ge d/\varepsilon^2\), the one-sided condition appearing in the probability estimate in Proposition~\ref{main} is equivalent to the corresponding two-sided condition.
In Section~\ref{numerical}, we compare the theoretical predictions
of Proposition~\ref{main} with empirical results obtained
from numerical experiments.

\begin{remark}\label{remark:alternative_concentration_proof}
An alternative concentration estimate can be obtained by applying Theorem~\ref{ricci} to each coordinate separately and then using a union bound. Define
\[
g_i:\mathrm{St}(n,d)\to \mathbb{R},
\qquad
g_i(U)=e_i^*\operatorname{diag}(UU^*),
\]
so that \(g_i(U)=\|u_i\|^2\). Then \(\mathbb{E}[g_i(U)]=d/n\), and \(g_i\) is \(2\sqrt2\)-Lipschitz with respect to the geodesic distance. Applying Theorem~\ref{ricci} and summing over \(i\) yields
\[
\mathbb{P}\!\left(
\exists\, i,\;
\left|\|u_i\|^2-\frac dn\right|\ge \varepsilon
\right)
\le
2n\exp\!\left(-\frac{(n-2)\varepsilon^2}{32}\right).
\]
This bound is valid, but weaker than Theorem~\ref{thm:prevalence_epsH} because of the extra factor \(n\) produced by the union bound.
\end{remark}

\section{Application to the Paulsen Problem}\label{sec:paulsen_problem}
In this section, we present an application of  Theorem~\ref{concent_sphere} to the \emph{Paulsen problem}, a major open problem in finite-dimensional frame theory~\cite{bodmann2010road,casazza2013kadison,cahill2013paulsen,kwok2018paulsen,hamilton2021paulsen}, which we now recall. We call \((x_i)_{i=1}^n\) an \textit{\(\varepsilon\)-nearly Parseval frame} if \(\|S-I_d\|_{\mathrm{op}} \leq \varepsilon\), and an 
\textit{\(\varepsilon\)-nearly equal-norm frame} if the squared norms of its frame vectors satisfy
\((1-\varepsilon)\frac{d}{n}
\leq
\|x_i\|^2
\leq
(1+\varepsilon)\frac{d}{n}\) for all \(i \in \{1,\ldots,n\}\). We write \(\epsF\) for the class of frames that are both \(\varepsilon\)-nearly equal-norm and \(\varepsilon\)-nearly Parseval, and \(\F\) for the class of (exactly) equal-norm Parseval frames in \(\mathbb{R}^d\).

\begin{problem}[Paulsen Problem~\cite{bodmann2010road}]\label{prob:Paulsen}
    Find the simplest function $f(d,n,\varepsilon)$ satisfying, for all $(x_i)_{i=1}^n \in \epsF$, the bound 
    \[
    \inf_{(y_i)_{i=1}^n \in \F} d\big((x_i)_{i=1}^n,(y_i)_{i=1}^n \big)^2 \leq f(d,n,\varepsilon),
    \]
    where the distance is 
    \[
    d\big((x_i)_{i=1}^n,(y_i)_{i=1}^n \big) = \left(\sum_{i=1}^n \|x_i - y_i\|^2 \right)^{1/2}.
    \]
\end{problem}

The existence of such a function $f(d,n,\varepsilon)$ is attributed to Hadwin
by Bodmann and Casazza~\cite{bodmann2010road}. 
Subsequent work focused on expressing a bounding function in the cleanest possible form and, in particular, without explicit dependence on $n$. 
In \cite{kwok2018paulsen}, Kwok et al.\ were the first to obtain a bound depending only on $d$ and $\varepsilon$; more precisely, they proved a bound of order $O(\varepsilon d^{13/2})$. 
Subsequently, Hamilton and Moitra obtained the following result, which will be useful for our purposes.

\begin{theorem}[{\cite{hamilton2021paulsen}}]\label{thm:hamilton_moitra}
    For $(x_i)_{i=1}^n \in \epsF$, 
    \[
    \inf_{(y_i)_{i=1}^n \in \F} d\big((x_i)_{i=1}^n,(y_i)_{i=1}^n \big)^2 \leq 20\varepsilon d^2.
    \]
\end{theorem}

At present, this is the best general bound for the Paulsen problem in the literature. On the other hand, it was shown in \cite{casazza2013kadison} that any such bound must necessarily be at least of order $\varepsilon^2 d$. Our main result is the following. In the statement, we refer to a \emph{uniform measure} on $\epsF$, which is defined precisely in the following subsection---see \eqref{eqn:measure_on_F}.

\begin{theorem}[Concentration for the Paulsen Problem] \label{newpaulsen}
    Let $\varepsilon \in (0,1)$ and let $(x_i)_{i=1}^n$ be chosen uniformly at random from $\epsF$. Then the bound
    \begin{equation}\label{eqn:newPaulsenBound}
    \inf_{(y_i)_{i=1}^n \in \F} d\big((x_i)_{i=1}^n,(y_i)_{i=1}^n \big)^2 \leq (\sqrt{20}+\sqrt{2})^2\varepsilon^2 d
    \end{equation}
    holds with  probability at least $$1 - \frac{2d \exp\!\left(
-\,
\frac{n \varepsilon^4}{
2d^3
}
\right)}{1
-
2d \exp\!\left(
-\frac{6 {\varepsilon}^2 n}{5(d+2)^2}
\right)}.$$
\end{theorem}

The next (immediate) corollary quantifies the success probability in terms of $n$.

\begin{corollary}
A random frame in $\epsF$ satisfies the Paulsen bound \eqref{eqn:newPaulsenBound} with probability $95\%$ if
\[
n \;\ge\; \frac{2d^3 \ln(80d)}{ \varepsilon^4},
\]
and with probability $99\%$ if
\[
n \;\ge\; \frac{2d^3 \ln(400d)}{ \varepsilon^4}.
\]
\end{corollary}

The theorem is proven in the next subsection, after setting up necessary preliminary results. In Section~\ref{sec:optimal_Paulsen_Bound}, we clarify the distinction between the best possible deterministic scaling and the high-probability scaling obtained here.

\subsection{Proof of  Theorem \ref{newpaulsen}}

Let us begin by setting up some relevant notation. Consider the subset of $\mathbb{R}^d$ consisting of all vectors whose squared norms lie between $(1-\varepsilon)\frac{d}{n}$ and $(1+\varepsilon)\frac{d}{n}$.
This set can be viewed as a thin spherical shell in $\mathbb{R}^d$.  
We denote it by
\[
\mathcal{S}_{\varepsilon}^d
\coloneqq
\left\{
v \in \mathbb{R}^d
:\;
(1-\varepsilon)\frac{d}{n}
\leq
\|v\|^2
\leq
(1+\varepsilon)\frac{d}{n}
\right\}.
\]
Let $\mu$ denote the uniform probability measure on
$\mathcal{S}^d_{\varepsilon}$, and consider the product measure
$\mu^{\otimes n}$ on $(\mathcal{S}^d_{\varepsilon})^n$.
Our primary set of interest is $\epsF \subset (\mathcal{S}^d_{\varepsilon})^n$. Since $\mathbb{P}_{\mu^{\otimes n}}\big( (x_1,\dots,x_n) \in \epsF \big) > 0$, the corresponding conditional probabilities are well-defined. Let $\nu$ denote the conditional probability measure, i.e.,
\begin{equation}\label{eqn:measure_on_F}
\nu(\,\cdot \,) \coloneqq \mu^{\otimes n}(\,\cdot \mid \epsF).
\end{equation}

Moreover, $\nu$ is precisely the normalized restriction of the uniform product measure
to $\epsF$; in particular, it is again uniform on $\epsF$. This is the measure on $\epsF$ that is used in Theorem \ref{newpaulsen}.

The proof of Theorem~\ref{newpaulsen} consists of two separate parts. In the first part, we compute the relevant probability, and in the second part, we derive the upper bound. Since the proof is rather lengthy, we first establish some auxiliary results that will be needed in the argument.

\begin{lemma}{\label{expectedvalueshell}}
Let $v$ be a random vector uniformly distributed on $\mathcal{S}^d_{\varepsilon}$. Then, 
    \[
\mathbb{E}\|v\|^2
=
\frac dn
\cdot
\frac{d}{d+2}
\cdot
\frac{
(1+\varepsilon)^{\frac{d+2}{2}}
-
(1-\varepsilon)^{\frac{d+2}{2}}
}{
(1+\varepsilon)^{\frac d2}
-
(1-\varepsilon)^{\frac d2}
} \leq \frac dn
\cdot
\frac{d}{d+2}
\cdot \left( 1+\varepsilon +\frac{2(1-\varepsilon)}{d}\right).
\]
\end{lemma}

\begin{proof} By the definition of the expected value, we have
    \[
\mathbb{E}\|v\|^2
=
\frac{\int_{\mathcal{S}^d_{\varepsilon}} \|v\|^2\, dv}{\int_{\mathcal{S}^d_{\varepsilon}} dv}.
\]
Using polar coordinates $v=r\omega$ with $\omega\in S^{d-1}$ and $dv = r^{d-1}\,dr\,d\omega$, and writing
\[
a=\sqrt{\frac dn(1-\varepsilon)},
\qquad
b=\sqrt{\frac dn(1+\varepsilon)},
\]
we obtain
\[
\mathbb{E}\|v\|^2
=
\frac{\int_a^b r^2\, r^{d-1}\,dr}{\int_a^b r^{d-1}\,dr}
=
\frac{d}{d+2}\cdot
\frac{b^{d+2}-a^{d+2}}{b^{d}-a^{d}}.
\]
Substituting $a,b$ yields 
\begin{equation}\label{eqn:expectation}
\mathbb{E}\|v\|^2
=
\frac dn
\cdot
\frac{d}{d+2}
\cdot
\frac{
(1+\varepsilon)^{\frac{d+2}{2}}
-
(1-\varepsilon)^{\frac{d+2}{2}}
}{
(1+\varepsilon)^{\frac d2}
-
(1-\varepsilon)^{\frac d2}
},
\end{equation}
and this proves the equality part of the lemma. 

Since the expression \eqref{eqn:expectation} is  complicated, we now derive a simpler upper bound. Let $k=d/2$, $\alpha=1+\varepsilon$ and $\beta=1-\varepsilon$ where $\varepsilon \in (0,1)$. Define
\[
L
\coloneqq
\frac{\alpha^{k+1}-\beta^{k+1}}{\alpha^{k}-\beta^{k}}.
\]

Observe that $\alpha^{k+1}-\beta^{k+1}
=
\alpha(\alpha^k-\beta^k)
+
\beta^k(\alpha-\beta)$, hence
\[
L
=
\alpha+\frac{\beta^k(\alpha-\beta)}{\alpha^k-\beta^k}
=
\alpha+\frac{\alpha-\beta}{\left(\frac{\alpha}{\beta}\right)^k-1}.
\]

Now
\[
\alpha-\beta=2\varepsilon,
\qquad
\frac{\alpha}{\beta}
=
\frac{1+\varepsilon}{1-\varepsilon}
=
1+u, \quad \text{where} \quad u=\frac{2\varepsilon}{1-\varepsilon}>0.
\]

Then
\[
\left(\frac{\alpha}{\beta}\right)^k-1
=
(1+u)^k-1
\ge
ku \quad \text{by Bernoulli's inequality.}
\]

Therefore,
\[
\frac{\alpha-\beta}{\left(\frac{\alpha}{\beta}\right)^k-1}
\le
\frac{2\varepsilon}{k\cdot \frac{2\varepsilon}{1-\varepsilon}}
=
\frac{1-\varepsilon}{k}.
\]

Thus
\[
L \le \alpha+\frac{1-\varepsilon}{k}.
\]
This implies the inequality in the lemma. 
\end{proof}

Using Lemma~\ref{expectedvalueshell}, we obtain the following result.

\begin{proposition}\label{propshell}
Let $x_1,\dots,x_n$ be i.i.d.\ sampled from $\mathcal{S}^d_{\varepsilon}$. Then
\[
\mathbb{P}_{\mu^{\otimes n}}\big( (x_1,\dots,x_n) \in \epsF \big)
>
1
-
2d \exp\!\left(
-\frac{24 {\varepsilon}^2 n}{(d+2)(1+\varepsilon)\left(3(d+2)(1+\varepsilon)+4d\varepsilon\right)}
\right).
\]
\end{proposition}

\begin{proof}
We have 
\begin{align*}
\mathbb{P}_{\mu^{\otimes n}}\big( (x_1,\dots,x_n) \in \epsF \big) &= \mathbb{P}_{\mu^{\otimes n}}\Big(
(x_1,\dots,x_n)
\text{ is $\varepsilon$–nearly Parseval}
\Big)\\
&= \mathbb{P}_{\mu^{\otimes n}}\Big( \|S-I_d\|_{\mathrm{op}}\leq \varepsilon
\Big)\\
&= 1 - \mathbb{P}_{\mu^{\otimes n}}\Big( \|S-I_d\|_{\mathrm{op}} > \varepsilon
\Big)
\end{align*}

Proceeding as in the sphere and ball cases, we invoke
Theorem~\ref{bernstein} , the Matrix Bernstein inequality. Define $Y_i \coloneqq x_i x_i^{*} - \mathbb{E}\!\left[x_i x_i^{*}\right]$,
so that, clearly $\mathbb{E}[Y_i]=0$, and $\mathbb{E}\!\left[x_i x_i^{*}\right]= (\mathbb{E}\|v\|^2/d) I_d$ by the rotation invariant distribution. Next, we estimate the parameters in Theorem~\ref{bernstein}.

First, we need to find $R$, an upper bound for $\|Y_i\|_{\mathrm{op}}$. Since $\|x_i\|^2 \le \frac dn(1+\varepsilon)$ for all
$x_i \in \mathcal{S}^d_{\varepsilon}$, we have
\[
\|x_i x_i^{*}\|_{\mathrm{op}}=\|x_i\|^2
\le
\frac dn(1+\varepsilon).
\]
Therefore,
\begin{align*}
\|Y_i\|_{\mathrm{op}}
=
\left\|x_i x_i^{*}- \left(\frac{\mathbb{E}\|v\|^2}{d}\right) I_d\right\|_{\mathrm{op}}
&=
\frac{d}{n}(1+\varepsilon) - \frac{\mathbb{E}\|v\|^2}{d} \leq \frac dn(1+\varepsilon),
\end{align*}
 so we can take the last term as $R$.

Now we derive an expression for the variance parameter $\sigma^2$. We compute 
\begin{align*}
\sum_{i=1}^n \mathbb{E}[Y_i^2]
=
n \mathbb{E}[Y_i^2]&=n\mathbb{E}\left[\|x_i\|^2x_ix_i^* - 2\left(\frac{\mathbb{E}\|x_i\|^2}{d}\right) x_i x_i^{*}+ \left(\frac{\mathbb{E}\|x_i\|^2}{d}\right)^2I_d\right]\\
&=n\left(\frac{\mathbb{E}\left[\|x_i\|^4\right]}{d}-\frac{\left(\mathbb{E}\|x_i\|^2\right)^2}{d^2}\right)I_d.
\end{align*}
Taking operator norms and using
$\|x_i\|^2 \le \frac dn(1+\varepsilon)$, we obtain
\[
\sigma^2
\coloneqq
\left\|
\sum_{i=1}^n \mathbb{E}[Y_i^2]
\right\|_{\mathrm{op}}
\le n \left\|\frac{\mathbb{E}\left[\|x_i\|^4\right]}{d}I_d\right\|_{\mathrm{op}} \le
n \cdot
\frac{\big(\frac dn(1+\varepsilon)\big)^2}{d}
=
\frac dn(1+\varepsilon)^2.
\]

Therefore, we obtain
\[
\mathbb{P}\Big(\big\|S - M I_d\big\|_{\mathrm{op}} > t\Big)
\le
2d \exp\!\left(
-\frac{3 t^2 n}{6d(1+\varepsilon)^2 + 2d(1+\varepsilon)t}
\right),
\]
where $M \coloneqq n\left(\mathbb{E}[\|v\|^2]/d\right)$. By the triangle inequality for the operator norm, if
\[
\|S-I_d\|_{\mathrm{op}}-\|I_d-MI_d\|_{\mathrm{op}} > t,
\]
then
\[
\|S-MI_d\|_{\mathrm{op}} > t.
\]
Hence, the probability of the first event is bounded by the probability of
the second, and the previous tail bound also applies to
\[
\mathbb{P}\Big(\|S-I_d\|_{\mathrm{op}} > t+\|I_d-MI_d\|_{\mathrm{op}}\Big).
\]

Since our aim is an $\varepsilon$--nearly Parseval statement, we set
\[
\varepsilon = t+\|I_d-MI_d\|_{\mathrm{op}},
\qquad\text{i.e.,}\qquad
t=\varepsilon-\|I_d-MI_d\|_{\mathrm{op}}.
\]
To use the bound we must have $t>0$, so it remains to ensure
\[
\varepsilon-\|I_d-MI_d\|_{\mathrm{op}} = \varepsilon - |1 - M| > 0.
\]

By Lemma~\ref{expectedvalueshell}, we know the exact value and an upper bound for $M$
$$M=\frac{d}{d+2}
\cdot
\frac{
(1+\varepsilon)^{\frac{d+2}{2}}
-
(1-\varepsilon)^{\frac{d+2}{2}}
}{
(1+\varepsilon)^{\frac d2}
-
(1-\varepsilon)^{\frac d2}
} \leq
\frac{d}{d+2}
\cdot \left( 1+\varepsilon +\frac{2(1-\varepsilon)}{d}\right).$$

Viewing $M$ as a function of $d$, one checks that it is increasing in $d$ and
satisfies $M=1$ at $d=2$. Hence, for $d\ge 2$ we have $M\ge 1$, and therefore
\[
\|I_d - M I_d\|_{\mathrm{op}} = |1-M| = M-1.
\]
In particular,
\[
t
=
\varepsilon - \|I_d - M I_d\|_{\mathrm{op}}
=
\varepsilon - (M-1)
=
\varepsilon + 1 - M \geq \frac{4\varepsilon}{d+2} > 0.
\]

Substituting all of the parameters we have obtained into
Proposition~\ref{bernstein} yields the claimed result.

\end{proof}

We now have all the necessary ingredients to prove
Theorem~\ref{newpaulsen}.

\begin{proof}[Proof of Theorem~\ref{newpaulsen}]
    We define the radial projection map
\[
P : \mathcal{S}_{\varepsilon}^d
\longrightarrow
\sqrt{\frac{d}{n}}\, \mathbb{S}^{d-1} \quad \text{by} \quad P(v)
=
\frac{\sqrt{d/n}}{\|v\|} \, v .
\]

This map rescales each vector in the shell onto the sphere of radius $\sqrt{d/n}$. Since $\mu$ is rotation-invariant and $P(Qv) = QP(v)$ for every orthogonal transformation $Q \in O(d)$, it follows that the pushforward measure $P_{\#}\mu$ is also rotation-invariant on the sphere $\sqrt{d/n}\, \mathbb{S}^{d-1}$. By uniqueness of the normalized Haar measure on the sphere, we conclude that $P_{\#}\mu$ is precisely the uniform probability measure on $\sqrt{d/n}\, \mathbb{S}^{d-1}$.

The strategy for the rest of the proof is sketched as follows. After selecting $(x_i)_{i=1}^n \in \epsF$, we show that the projected vectors
$(P(x_i))_{i=1}^n$ are, with high probability,
$\varepsilon^2/d$--nearly Parseval.
Since they also have equal-norm with $\|P(x_i)\|=\sqrt{d/n}$, they belong to $\mathcal{F}_{d,n}^{\varepsilon^2/d}$, so that we can apply the result of Hamilton and Moitra
(Theorem~\ref{thm:hamilton_moitra}),
we obtain a sharper upper bound.
Moreover, we show that $(x_i)_{i=1}^n$ is close to
$(P(x_i))_{i=1}^n$.
We will then apply the triangle inequality to conclude the proof.

We begin by determining the relevant probability.
For this purpose, we introduce some notation and definitions.

\[
\mathcal{A}
\coloneqq
\left\{
\text{$\frac{\varepsilon^2}{d}$–nearly Parseval frames
with vectors in }
\sqrt{\frac{d}{n}}\, \mathbb{S}^{d-1}
\right\}.
\]

We are interested in the conditional probability $\mathbb{P}_{\nu}\Big(
(P(x_1),\dots,P(x_n))\in \mathcal{A}
\Big)$. Since $\epsF$ corresponds exactly to the
$\varepsilon$–nearly Parseval condition in $\bigl(\mathcal{S}_\varepsilon^d\bigr)^n$, this equals
\[
\mathbb{P}_{\mu^{\otimes n}}\Big(
(P(x_1),\dots,P(x_n))\in \mathcal{A}
\;\Big|\;
(x_1,\dots,x_n)
\text{ is $\varepsilon$–nearly Parseval}
\Big).
\]

Using complements, we write
\[
=
1
-
\mathbb{P}_{\mu^{\otimes n}}\Big(
(P(x_1),\dots,P(x_n))\notin \mathcal{A}
\;\Big|\;
(x_1,\dots,x_n)
\text{ is $\varepsilon$–nearly Parseval}
\Big).
\]

By the definition of conditional probability,
\[
=
1
-
\frac{
\mathbb{P}_{\mu^{\otimes n}}\Big(
(P(x_1),\dots,P(x_n))\notin \mathcal{A}
\;\text{ and }\;
(x_1,\dots,x_n)
\text{ are $\varepsilon$–nearly Parseval}
\Big)
}{
\mathbb{P}_{\mu^{\otimes n}}\Big(
(x_1,\dots,x_n)
\text{ is $\varepsilon$–nearly Parseval}
\Big)
}.
\]

Finally, we obtain the lower bound
\[
\geq
1
-
\frac{
\mathbb{P}_{\mu^{\otimes n}}\Big(
(P(x_1),\dots,P(x_n))\notin \mathcal{A}
\Big)
}{
\mathbb{P}_{\mu^{\otimes n}}\Big(
(x_1,\dots,x_n)
\text{ is $\varepsilon$–nearly Parseval}
\Big)
}.
\]

Since $P_{\#}\mu$ is the uniform probability measure on
$\mathbb{S}^{d-1}_{\sqrt{d/n}}$, the vectors
$P(x_1),\dots,P(x_n)$ are i.i.d.\ uniform on the sphere.
Hence
\[
\mathbb{P}_{\mu^{\otimes n}}\Big((P(x_1),\dots,P(x_n))\notin \mathcal{A}\Big)
\]
is exactly the probability in Theorem~\ref{concent_sphere}
with $\varepsilon$ replaced by $\varepsilon^2/d$.

Therefore, 
\[
\mathbb{P}_{\mu^{\otimes n}}\Big((P(x_1),\dots,P(x_n))\notin \mathcal{A}\Big)
\le
2d \exp\!\left(
-\,
\frac{n \varepsilon^4}{
2d\!\left(d^2-d+\frac{2\varepsilon^2(d-1)}{3}\right)
}
\right).
\]

By Proposition~\ref{propshell}, we have
$$\mathbb{P}_{\mu^{\otimes n}}\Big(
(x_1,\dots,x_n)
\text{ is $\varepsilon$–nearly Parseval}
\Big) > 1
-
2d \exp\!\left(
-\frac{24 {\varepsilon}^2 n}{(d+2)(1+\varepsilon)\left(3(d+2)(1+\varepsilon)+4d\varepsilon\right)}
\right)$$

Consequently, we obtain the following lower bound for the main probability:

\[
\mathbb{P}_{\nu}\Big(
(P(x_1),\dots,P(x_n))\in \mathcal{A}
\Big) \geq 1 - \frac{2d \exp\!\left(
-\,
\frac{n \varepsilon^4}{
2d \,\!\left(d^2-d+\frac{2\varepsilon^2(d-1)}{3}\right)
}
\right)}{1
-
2d \exp\!\left(
-\frac{24 {\varepsilon}^2 n}{(d+2)(1+\varepsilon)\left(3(d+2)(1+\varepsilon)+4d\varepsilon\right)}
\right)}
\]

After simplifying the expression in this probability bound,
we may write the following alternative lower bound.

$$\mathbb{P}_{\nu}\Big(
(P(x_1),\dots,P(x_n))\in \mathcal{A}
\Big) \geq 1 - \frac{2d \exp\!\left(
-\,
\frac{n \varepsilon^4}{
2d^3
}
\right)}{1
-
2d \exp\!\left(
-\frac{6 {\varepsilon}^2 n}{5(d+2)^2}
\right)}$$

We have now completed the analysis of the probability term appearing
in the theorem. In particular, for a randomly chosen
$(x_i)_{i=1}^n \in \epsF$, we have obtained a lower bound on the probability that $(z_i)_{i=1}^n$ with $z_i = P x_i$ forms an $\varepsilon^2/d$--nearly Parseval frame. Moreover, since each $z_i$ satisfies $\|z_i\| = \sqrt{d/n}$, the vectors automatically have equal-norm. Consequently,
$(z_i)_{i=1}^n$ belongs to
$\mathcal{F}_{d,n}^{\varepsilon^2/d}$ with at least the same probability. In this setting, applying 
Theorem~\ref{thm:hamilton_moitra} of Hamilton and Moitra yields
\[
\inf_{(y_i)_{i=1}^n \in \F}
d\!\left((y_i)_{i=1}^n,(z_i)_{i=1}^n\right)^2
\;\le\; 20\frac{\varepsilon^2}{d} d^2 = 20\varepsilon^2d.
\]

At this point, our main focus is to estimate $d\!\left((x_i)_{i=1}^n,(z_i)_{i=1}^n\right)^2$. We have
\[
x_i - z_i
= x_i\!\left(1 - \sqrt{\frac{d}{n}}\,\frac{1}{\|x_i\|}\right),
\]
so that
\[
\|x_i - z_i\|^2
= \|x_i\|^2
\left(1 - \sqrt{\frac{d}{n}}\,\frac{1}{\|x_i\|}\right)^2
=
\left(\|x_i\| - \sqrt{\frac{d}{n}}\right)^2.
\]
Summing over $i$, we have
\[
\sum_{i=1}^n \|x_i - z_i\|^2
=
\sum_{i=1}^n
\left(\|x_i\| - \sqrt{\frac{d}{n}}\right)^2.
\]
Note that
\[
\|x_i\|
\in
\left[
\sqrt{\frac{d}{n}}\sqrt{1-\varepsilon},
\,
\sqrt{\frac{d}{n}}\sqrt{1+\varepsilon}
\right], 
\]
which implies 
\[
\left|\|x_i\| - \sqrt{\frac{d}{n}}\right|
\le
\sqrt{\frac{d}{n}}
\left(
\sqrt{1+\varepsilon}
-
\sqrt{1-\varepsilon}
\right),
\]
yielding
\[
\sum_{i=1}^n \|x_i - z_i\|^2
\le
\sum_{i=1}^n
\frac{d}{n}
\left(
\sqrt{1+\varepsilon}
-
\sqrt{1-\varepsilon}
\right)^2
=
d
\left(
\sqrt{1+\varepsilon}
-
\sqrt{1-\varepsilon}
\right)^2.
\]
Next, observe that 
\[
\left(
\sqrt{1+\varepsilon}
-
\sqrt{1-\varepsilon}
\right)^2
=
2 - 2\sqrt{1-\varepsilon^2}
<
2 - 2(1-\varepsilon^2)
=
2\varepsilon^2,
\]
which implies
\[
d\!\left((x_i)_{i=1}^n,(z_i)_{i=1}^n\right)^2=\sum_{i=1}^n \|x_i - z_i\|^2
\le
2\varepsilon^2 d.
\]
By using triangle inequality, we have
\begin{align*}
    \inf_{(y_i)_{i=1}^n \in \F}d\!\left((x_i)_{i=1}^n,(y_i)_{i=1}^n\right) &\leq d\!\left((x_i)_{i=1}^n,(z_i)_{i=1}^n\right) + \inf_{(y_i)_{i=1}^n \in \F}d\!\left((y_i)_{i=1}^n,(z_i)_{i=1}^n\right) \\
    &\leq \sqrt{2\varepsilon^2 d} + \sqrt{20\varepsilon^2 d} = \left( \sqrt{2}+\sqrt{20} \right)\sqrt{\varepsilon^2 d}.
\end{align*}
After squaring both sides, we obtain the result.
\end{proof}

It was shown by Cahill and Casazza~\cite{cahill2013paulsen} that the Paulsen problem is equivalent to the \emph{projection problem} in operator theory. Using this equivalence, we have the following probabilistic result for the projection problem. The proof follows directly from~\cite[Theorem 4.1]{cahill2013paulsen} and Theorem~\ref{newpaulsen}.

\begin{corollary}
    Let $e_1,\ldots,e_n \in \R^n$ be the standard basis vectors. Let $P$ be a random projection to a $d$-dimensional subspace, sampled under the condition that 
    \[
    (1-\varepsilon)\frac{d}{n} \leq \|Pe_i\|^2 \leq (1+\varepsilon)\frac{d}{n},
    \]
    for all $i$, where $\|\cdot\|$ is the standard norm. Then there exists a projection $Q$ with $\|Qe_i\|^2 = \frac{d}{n}$, for all $i$, satisfying
    \[
    \sum_{i=1}^n \|Pe_i - Qe_i\|^2 \leq 4\left(\sqrt{20}+\sqrt{2}\right)^2 \varepsilon^2 d
    \]
    with high probability, provided $n$ is sufficiently large. 
\end{corollary}

\subsection{On the Optimal Paulsen Bound}\label{sec:optimal_Paulsen_Bound}

The goal of this subsection is to clarify the theoretical order of an optimal bound $f(\varepsilon,d,n)$ in the Paulsen problem. As was mentioned above, it was shown in \cite{casazza2013kadison}, via a family of examples, that the Paulsen problem must at least be on the order of $\varepsilon^2 d$. On the other hand, quoting from~\cite{cahill2013paulsen}: ``For all examples we know at this time, we have
\(
f(\varepsilon,d,n) \leq 16\varepsilon d
\).''
The papers \cite{kwok2018paulsen} and \cite{hamilton2021paulsen} 
state that the optimal theoretical bound in the Paulsen problem must be at least of order $\varepsilon d$, both citing~\cite{cahill2013paulsen} as the source. However, as far as we can tell, there is no such theoretical bound established in that paper, or anywhere else in the literature. Our goal is to clarify the situation by proving:

\begin{proposition}\label{prop:paulsen_bound}
    If a bound $f(\varepsilon,d,n)$ for the Paulsen problem is a polynomial in $\varepsilon$ and $d$, then it must be of order at least $O(\varepsilon d)$. 
\end{proposition}

\begin{remark}
    Theorem~\ref{newpaulsen} gives a bound of the form $O(\varepsilon^2 d)$. Since the theorem is probabilistic, this does not contradict the proposition.
\end{remark}

The proof proceeds via a family of examples based on the example appearing in Equation~(38) of \cite{casazza2012auto}, which was later presented in greater detail, together with a proof, in Ramachandran's thesis~\cite{ramachandran2021geodesic}.

\begin{example}[{\cite[Example A.1.1]{ramachandran2021geodesic}}]
    For a fixed $\theta$ which is small enough, consider the synthesis operator $U$ where the columns represent a frame:
\[
U \coloneqq \frac{\sqrt{2}}{2}\begin{bmatrix} \cos(2\theta) & \cos(2\theta) & 0 & 0 \\ \sin(2\theta) & -\sin(2\theta) & 1 & 1 \end{bmatrix}.
\]
The corresponding frame of $U$ is in $\mathcal{F}_{2,4}^{\varepsilon}$ where $\varepsilon \leq 4 \sin^2\theta$.
The closest equal-norm Parseval frame (ENPF) to $U$ has synthesis operator
\[
V \coloneqq \frac{\sqrt{2}}{2}\begin{bmatrix} \cos(\theta) & \cos(-\theta) & \sin(-\theta) & \sin(\theta) \\ \sin(\theta) & \sin(-\theta) & \cos(-\theta) & \cos(\theta) \end{bmatrix},
\]
 and the squared distance is $\|U-V\|_{\mathrm{F}} \geq \sin^2\theta \geq \varepsilon/4$.
\end{example}

 \begin{remark}
     This example is exactly equal norm, and its frame operator is $I+O(\theta^2)$ for small values of $\theta$. Geometrically, this is saying that his family of examples is contained in the equal norm locus and is tangent to the Parseval locus. Thus this is a particularly structured example and in retrospect is should not be surprising that it manages to stay closer to the equal norm and Parseval locus than a generic perturbation.
 \end{remark}

By using this example, we generate the next one.

\begin{example}\label{ex2}
Let $k$ be a positive integer and let $d=2k$ and $n=4k$. Let $U_k$ and $V_k$ be the following matrices
\[U_k \coloneqq \begin{bmatrix} U & 0 & \cdots & 0 \\ 0 & U & \cdots & 0 \\ \vdots & \vdots & \ddots & \vdots \\ 0 & 0 & \cdots & U \end{bmatrix} \qquad \text{and} \qquad V_k\coloneqq \begin{bmatrix} V & 0 & \cdots & 0 \\ 0 & V & \cdots & 0 \\ \vdots & \vdots & \ddots & \vdots \\ 0 & 0 & \cdots & V \end{bmatrix}\]
where $U$ and $V$ are the matrices defined in the previous example. Let $\varepsilon$ be the same number from the previous example.
\end{example}

\begin{lemma}
The frames represented by $U_k$ and $V_k$ belong to $\epsF$ and $\F$, respectively.
\end{lemma}

\begin{proof}
The frame vectors represented by $U_k$ have the same nonzero components as those of $U$, 
with additional zero entries. In particular, each column of $U_k$ has the same norm as the 
corresponding column of $U$, so the equal-norm condition is preserved. Moreover, 
$U_k U_k^*$ is a block diagonal matrix with each diagonal block equal to $UU^*$, and 
therefore has the same eigenvalues as $UU^*$, each with multiplicity $k$. Since $U$ 
represents an $\varepsilon$-nearly equal-norm Parseval frame, all eigenvalues of $UU^*$ lie 
between $1 - \varepsilon$ and $1 + \varepsilon$, and the same holds for $U_k U_k^*$. 
Hence $U_k$ represents a frame in $\mathcal{F}_{d,n}^{\varepsilon}$. The same argument 
applied to $V$ and $V_k$ shows that $V_k$ represents a frame in $\mathcal{F}_{d,n}$.
\end{proof}

\begin{lemma}\label{lem:offblock}
Let $a_1, a_2 \in \mathbb{R}$ with $a_1^2 + a_2^2 = 1/2$, and let $b_1, b_2, \ldots, b_d \in \mathbb{R}$ 
with $b_1^2 + b_2^2 + \cdots + b_d^2 = 1/2$. Then the expression
\[
Y = (a_1 - b_1)^2 + (a_2 - b_2)^2 + \sum_{i=3}^{d} b_i^2
\]
is minimized only when $b_3 = \cdots = b_d = 0$.
\end{lemma}

\begin{proof}
This result follows from the Cauchy--Schwarz inequality.
\end{proof}

\begin{lemma}
The frame represented by $V_k$ is the closest ENPF to the frame represented by $U_k$.
\end{lemma}

\begin{proof}
Let $H$ be the synthesis operator of an arbitrary frame in $\mathcal{F}_{2k,4k}$. We decompose $H$ as 
$H = H^{\mathrm{diag}} + H^{\mathrm{off}}$, where $H^{\mathrm{diag}}$ retains the 
$2 \times 4$ diagonal blocks of $H$ and sets all other entries to zero, and $H^{\mathrm{off}}$ 
retains the remaining entries and sets the diagonal blocks to zero. Since $U_k$ is block diagonal, 
$U_k - H^{\mathrm{diag}}$ and $H^{\mathrm{off}}$ have disjoint support, and therefore
\[
\|U_k - H\|_{\mathrm{F}}^2 = \|U_k - H^{\mathrm{diag}}\|_{\mathrm{F}}^2 + \|H^{\mathrm{off}}\|_{\mathrm{F}}^2.
\]
Denote the columns of $U_k$ by $f_1, \ldots, f_{4k}$ and the columns of $H$ by $h_1, \ldots, h_{4k}$. 
For each $i$, write $h_i = h_i' + h_i''$, where $h_i'$ is the part supported on the diagonal 
block containing $f_i$ and $h_i''$ is the remaining part. Since $\|f_i\|^2 = \|h_i\|^2 = \frac{1}{2}$ and 
$\|h_i\|^2 = \|h_i'\|^2 + \|h_i''\|^2$, we have
\[
\|U_k - H\|_{\mathrm{F}}^2 = \sum_{i=1}^{4k} \left(\|f_i - h_i'\|^2 + \|h_i''\|^2\right).
\]
Each summand has the form treated in Lemma~\ref{lem:offblock}: the nonzero components of $f_i$ 
play the role of $(a_1, a_2)$ and the components of $h_i$ play the role of $(b_1, \ldots, b_d)$, 
all with fixed squared norm $\frac{1}{2}$. By the lemma, each summand is minimized only when $h_i'' = 0$. 
Therefore the minimizer $H$ must be block diagonal.

Since $H$ is block diagonal and represents an equal-norm Parseval frame, each $2 \times 4$ 
diagonal block $H_i$ of $H$ itself represents an ENPF. Therefore
\[
\|U_k - H\|_{\mathrm{F}}^2 = \sum_{i=1}^{k} \|H_i - U\|_{\mathrm{F}}^2 \geq k\,\|V - U\|_{\mathrm{F}}^2,
\]
where the inequality follows from the fact that the frame represented by $V$ is the closest ENPF to the frame represented by $U$. This lower bound is attained when 
$H = V_k$, so the frame represented by $V_k$ is the closest ENPF to the 
frame represented by $U_k$.
\end{proof}

Combining the above lemmas yields: 

\begin{proposition}
The minimum squared distance from the frame represented by $U_k$ to $\F$ satisfies
\[
\|V_k - U_k\|_{\mathrm{F}}^2 = k\,\|V - U\|_{\mathrm{F}}^2 \geq  k \cdot \frac{\varepsilon}{4} = \frac{d\,\varepsilon}{8}.
\]
In particular, for any $\varepsilon$-nearly ENPF in the regime $d = 2k$, $n = 4k$,  the distance to the closest ENPF grows linearly in $\varepsilon$, showing that the bound of order $\varepsilon^2 d$ cannot hold uniformly over all 
$\varepsilon$-nearly ENPFs.
\end{proposition}

This family of examples therefore proves Proposition~\ref{prop:paulsen_bound}.

\bibliographystyle{plain}
\bibliography{references}

\noindent \noindent \textbf{Samuel A.\  Ballas} \\
Department of Mathematics, Florida State University \\
\texttt{sballas@fsu.edu}

\vspace{1.5em}

\noindent \textbf{Ferhat Karabatman} \\
Department of Mathematics, Florida State University \\
\texttt{fk22@fsu.edu}

\vspace{1.5em}

\noindent \textbf{Tom Needham} \\
Department of Mathematics, Florida State University \\
\texttt{tneedham@fsu.edu}

\begin{appendix}
    
\section{Numerical Experiments}\label{numerical}

\begin{figure}
    \centering

    \begin{subfigure}{0.48\textwidth}
        \centering
        \includegraphics[width=\textwidth]{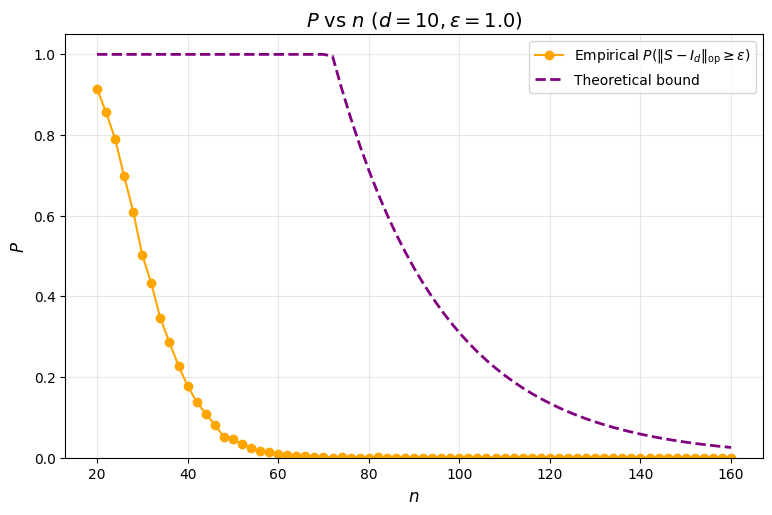}
    \end{subfigure}
    \hfill
    \begin{subfigure}{0.48\textwidth}
        \centering
        \includegraphics[width=\textwidth]{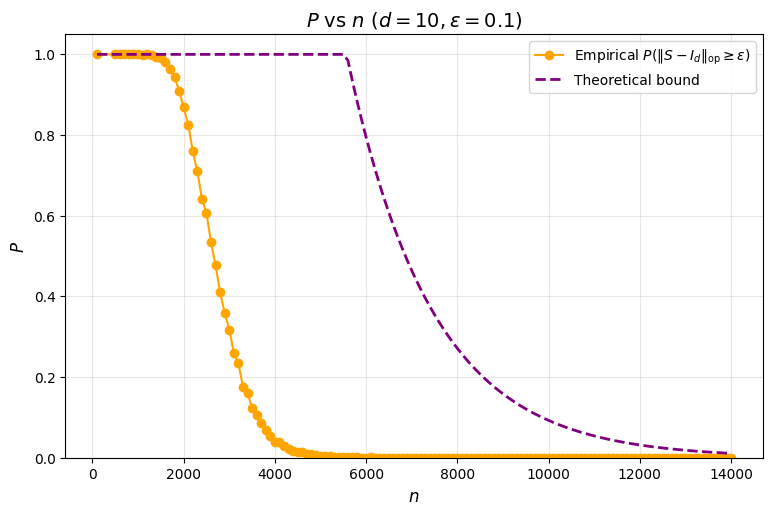}
    \end{subfigure}

    \vspace{0.5cm}

    \begin{subfigure}{0.48\textwidth}
        \centering
        \includegraphics[width=\textwidth]{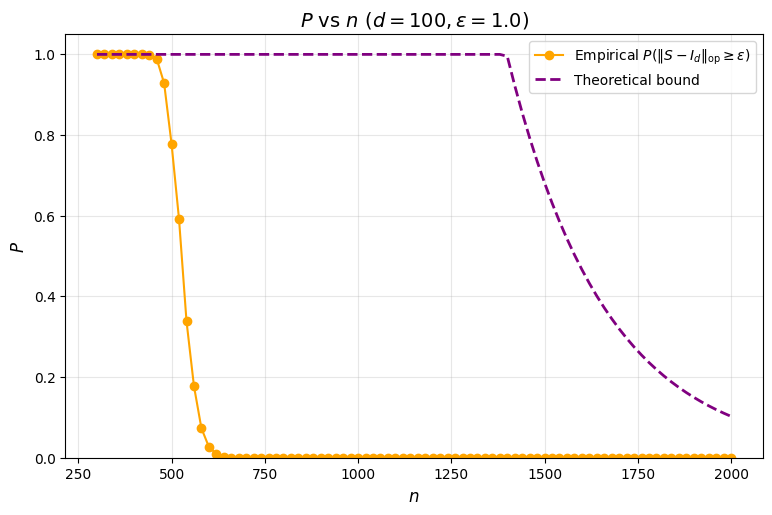}
    \end{subfigure}
    \hfill
    \begin{subfigure}{0.48\textwidth}
        \centering
        \includegraphics[width=\textwidth]{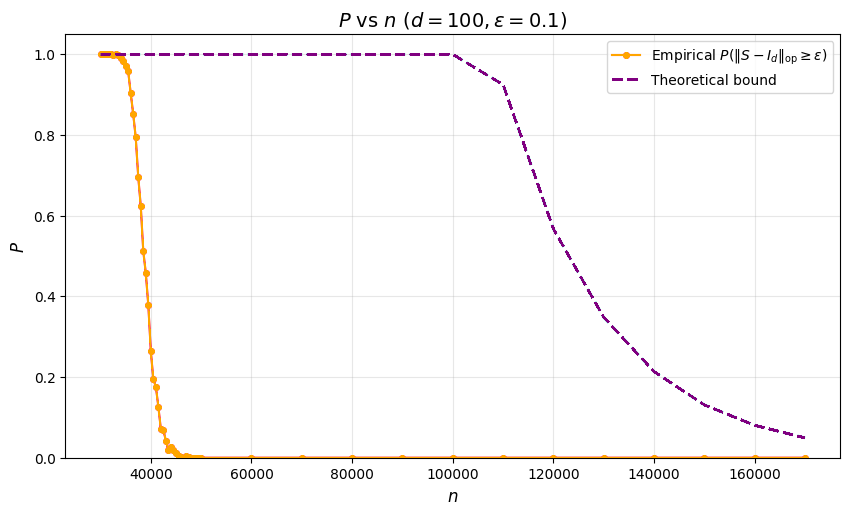}
    \end{subfigure}
\caption{Empirical probabilities of \(\|S-I_d\|_{\mathrm{op}} \geq \varepsilon\) for the sphere model, compared with the bound in Theorem~\ref{concent_sphere}}\label{figure_spheres}
\end{figure}

We numerically illustrate Theorem~\ref{concent_sphere} by sampling independent vectors uniformly from the sphere \(\mathbb S^{d-1}_{\sqrt{d/n}}\) and estimating the probability \(\mathbb P\bigl(\|S-I_d\|_{\mathrm{op}}\geq \varepsilon\bigr) \) via Monte Carlo simulation, as shown in Figure~\ref{figure_spheres}. The dashed orange curve represents the theoretical upper bound in Theorem~\ref{concent_sphere}, while the solid blue curve represents the empirical probability. In all cases, the empirical probability decays rapidly as \(n\) increases, illustrating the concentration of the random frame operator around the identity. The theoretical bound captures the same exponential decay behavior, although it is conservative, as expected from a non-asymptotic concentration inequality.

\begin{figure}
    \centering

    \begin{subfigure}{0.48\textwidth}
        \centering
        \includegraphics[width=\textwidth]{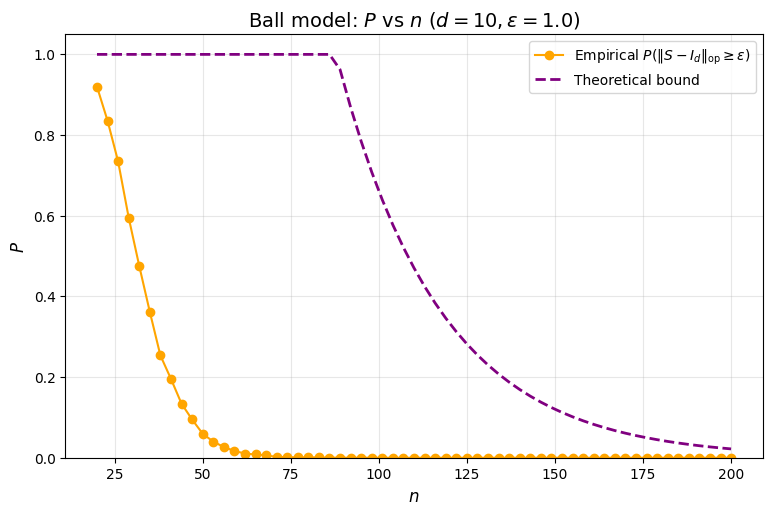}
    \end{subfigure}
    \hfill
    \begin{subfigure}{0.48\textwidth}
        \centering
        \includegraphics[width=\textwidth]{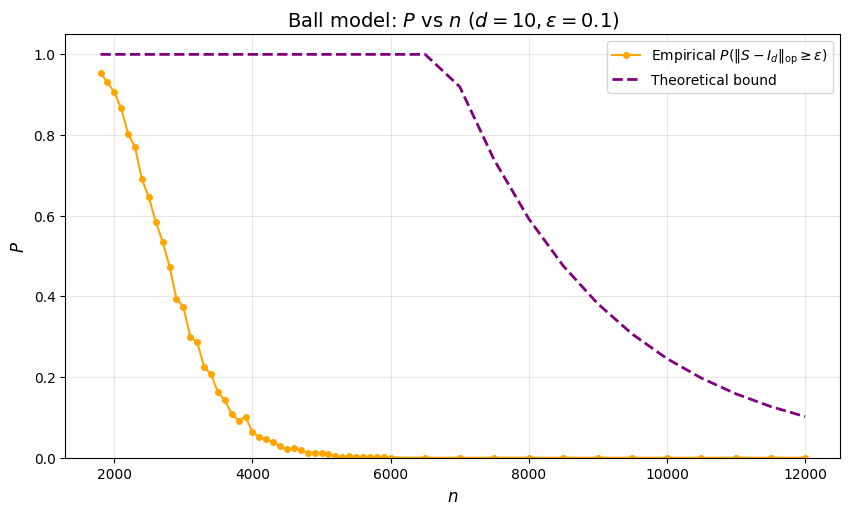}
    \end{subfigure}

    \vspace{0.5cm}

    \begin{subfigure}{0.48\textwidth}
        \centering
        \includegraphics[width=\textwidth]{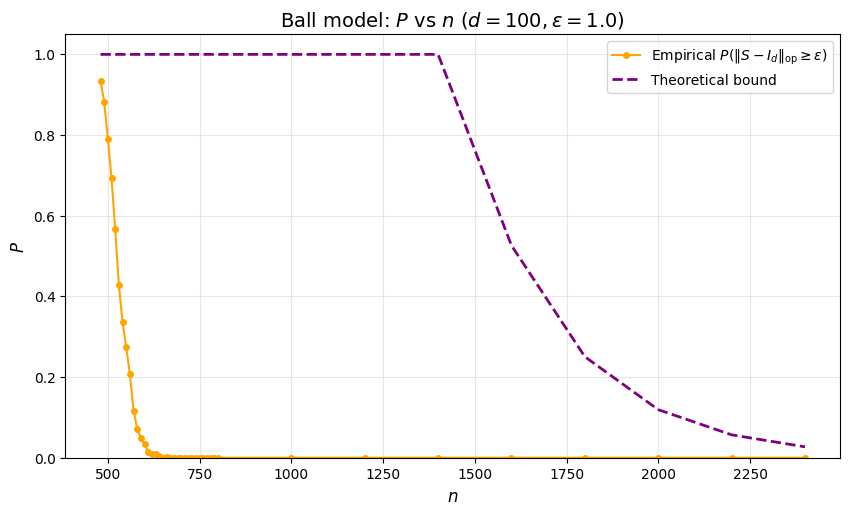}
    \end{subfigure}
    \hfill
    \begin{subfigure}{0.48\textwidth}
        \centering
        \includegraphics[width=\textwidth]{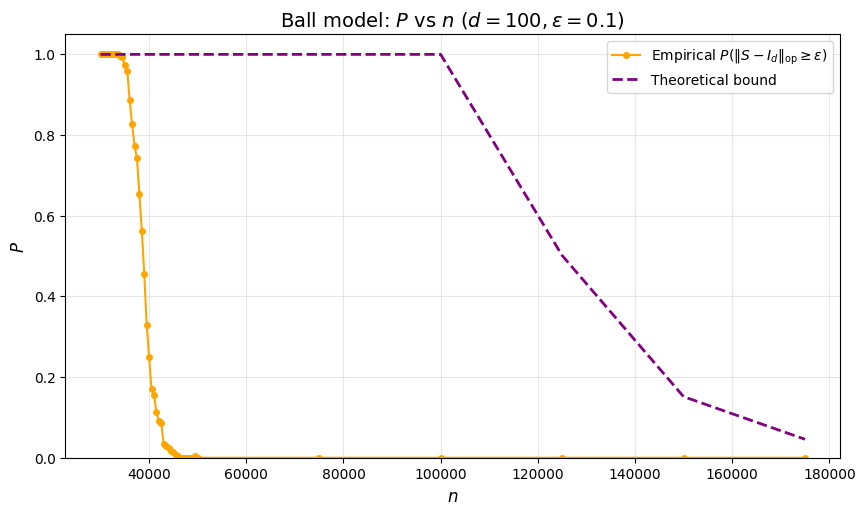}
    \end{subfigure}
\caption{Empirical probabilities of \(\|S-I_d\|_{\mathrm{op}} \geq \varepsilon\) for the ball model, compared with the bound in Theorem~\ref{concent_ball}}\label{figure_ball}
\end{figure}

We also illustrate the corresponding concentration result for the ball model. In Figure~\ref{figure_ball}, the vectors are sampled uniformly from the ball of radius \(\sqrt{(d+2)/n}\), which is chosen so that \(\mathbb E S=I_d\). As in the sphere case, the empirical probability \(\mathbb P\bigl(\|S-I_d\|_{\mathrm{op}}\geq \varepsilon\bigr)\)
decays rapidly as \(n\) increases. The theoretical bound captures the same decay behavior, while remaining conservative.

\begin{figure}
    \centering

    \begin{subfigure}{0.48\textwidth}
        \centering
        \includegraphics[width=\textwidth]{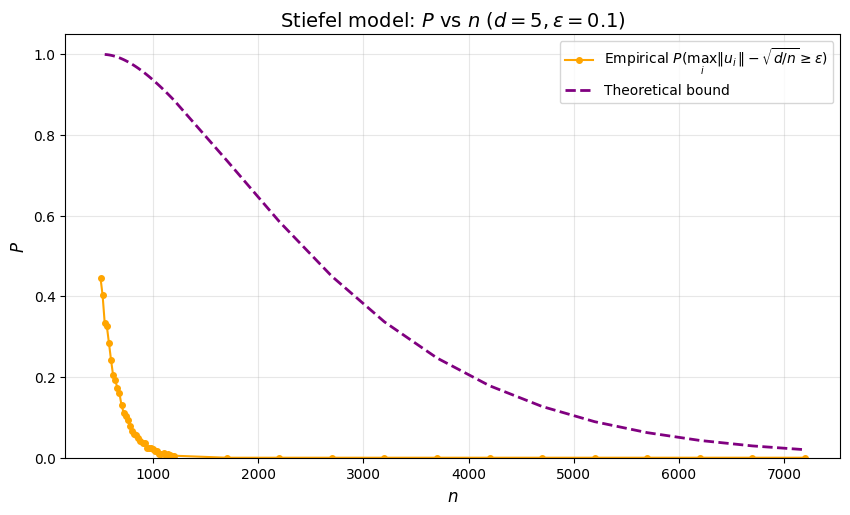}
    \end{subfigure}
    \hfill
    \begin{subfigure}{0.48\textwidth}
        \centering
        \includegraphics[width=\textwidth]{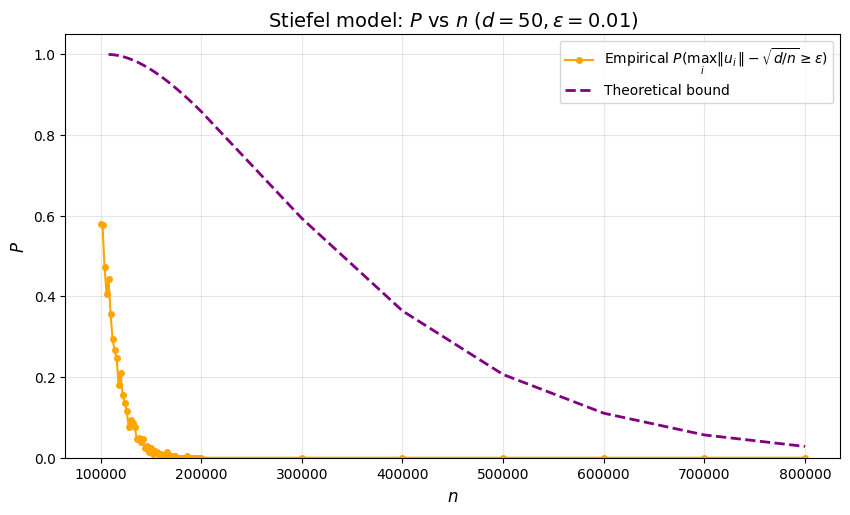}
    \end{subfigure}

\caption{Empirical probabilities for the Stiefel model, compared with the bound in Proposition~\ref{main}.}
\label{figure_stiefel}
\end{figure}

Finally, we numerically illustrate Proposition~\ref{main} by sampling matrices uniformly from the Stiefel manifold \(\mathrm{St}(n,d)\) and estimating the probability \(\mathbb P\left(
\max_{i\in\{1,\ldots,n\}}\|u_i\|-\sqrt{\frac dn}\geq \varepsilon
\right)\). Figure~\ref{figure_stiefel} compares the empirical probability with the theoretical upper bound.

In conclusion, our numerical experiments support the theoretical concentration bounds presented in the paper. However, the empirical behavior suggests that there may be room for improvement in the concentration inequalities. We leave this as a potential future research direction.

\clearpage
\end{appendix}

\end{document}